\documentclass{amsart}
\usepackage{amsmath}
\usepackage{amssymb}
\usepackage{mathrsfs}
\usepackage{pifont}
\usepackage{amsfonts}
\usepackage{stmaryrd}
\usepackage{latexsym,amssymb,amsmath}

\newtheorem{theorem}{Theorem}[section]
\newtheorem{lemma}[theorem]{Lemma}
\newtheorem{corollary}[theorem]{Corollary}
\newtheorem{proposition}[theorem]{Proposition}

\theoremstyle{definition}

\theoremstyle{remark}
\newtheorem{remark}[theorem]{Remark}

\numberwithin{equation}{section}

\begin{document}

\title{Character Formulae for Ortho-symplectic Lie Superalgebras $\mathfrak{osp}(n|2)$}

\author{Li Luo}
\address{Department of Mathematics, East China Normal University, Shanghai, 200241, China}
%%\curraddr{Department of Mathematics and Statistics, Case Western
%%Reserve University, Cleveland, Ohio 43403}
\email{lluo@math.ecnu.edu.cn}
%%\thanks{The first author was supported in part by NSF Grant \#000000.}

%\author{R. B. Zhang}
%\address{School of Mathematics and Statistics,
%University of Sydney, NSW 2006, Australia}
%\email{rzhang@maths.usyd.edu.au}
%%\thanks{Support information for the second author.}

%    General info

%%\date{January 23, 2008.}

%%\dedicatory{This paper is dedicated to our advisors.}


\begin{abstract}

The character formula of any finite dimensional irreducible module
$L_\lambda$ for Lie superalgebra $\mathfrak{osp}(n|2)$ is computed.
As a by-product, the decomposition of tensor module
$L_\lambda\otimes \mathbb{C}^{n|2}$, where $\mathbb{C}^{n|2}$ is the
natural representation, is obtained.

\end{abstract}

\maketitle

%%\section*{This is an unnumbered first-level section head}
%%This is an example of an unnumbered first-level heading.

%% The correct journal style for \specialsection is all uppercase; a known bug
%% in amsart.cls prevents this, so input must be uppercase until it is fixed.
%\specialsection*{This is a Special Section Head}
%%\specialsection*{THIS IS A SPECIAL SECTION HEAD}
%%This is an example of a special section head%
%%%%%%%%%%%%%%%%%%%%%%%%%%%%%%%%%%%%%%%%%%%%%%%%%%%%%%%%%%%%%%%%%%%%%%%%
%%\footnote{Here is an example of a footnote. Notice that this
%%footnote text is running on so that it can stand as an example of
%%how a footnote with separate paragraphs should be written.
%%\par
%%And here is the beginning of the second paragraph.}%
%%%%%%%%%%%%%%%%%%%%%%%%%%%%%%%%%%%%%%%%%%%%%%%%%%%%%%%%%%%%%%%%%%%%%%%%

\section{Introduction}

\subsection{}
In his foundational papers \cite{K1, K2, K3}, Kac classified the
finite dimensional complex simple Lie superalgebras and developed a
character formula for the so-called typical irreducible
representations. However, it turned out to be one of the most
challenging problems in the theory of Lie superalgebras to find the
character formulae of the so-called atypical irreducible
representatios. In the case of type $A$, it was first solved by
Serganova in \cite{Se}, where a Kazhdan-Lusztig theory was developed
for $\mathfrak{gl}(m|n)$. Later on, using quantum group techniques,
the Kazhdan-Lusztig polynomials were computed quite directly by
Brundan in \cite{B}. Reworking Brundan's description, Su and Zhang
\cite{SZ} gave a very explicit closed character formula. Besides,
van der Jeugt \cite{V} constructed a character formula for all
finite dimensional irreducible representations of ortho-symplectic
superalgebra $\mathfrak{osp}(2|2n)$ (i.e. type $C$).

However, there is still no complete character formula for a II type
basic classical Lie superalgebra (i.e. types $B$ and $D$) up to this
stage. The purpose of the present paper is to find such complete
character formulae for simple Lie superalgebras
$\mathfrak{osp}(2m+1|2)$ and $\mathfrak{osp}(2m|2)$ (i.e. types
$B(m|1)$ and $D(m|1)$, respectively). Precisely, we will describe
the character of any irreducible module in terms of the characters
of generalized Verma modules, which have been known clearly. We hope
that our effort may help us to explore the more general cases of II
type basic classical Lie superalgebras in the near future.

\subsection{}
In the theory of finite dimensional simple Lie algebras, it is well
known that all finite dimensional irreducible representations except
the spin representations can be found in tensor module
$V^p=\overbrace{V\otimes V\otimes\cdots\otimes V}^p$ for some $p$,
where $V$ is the natural representation. Since there cannot be spin
representation in the super case, it is possible for us to study all
finite dimensional irreducible $\mathfrak{osp}(n|2)$-modules via the
tensor module $(\mathbb{C}^{n|2})^p$. In other words, we can study
an irreducible module $L_\lambda$ from $L_\mu\otimes
\mathbb{C}^{n|2}$ where $L_\mu$ is another irreducible module known
by induction.

However, the tensor module $L_\mu\otimes \mathbb{C}^{n|2}$ may not
be completely reducible in the super case. Thus we should firstly
study the blocks of irreducible modules appearing in $L_\mu\otimes
\mathbb{C}^{n|2}$, and then have a more detailed argument for these
blocks.

\subsection{}
In order to calculate the tensor module $L_\mu\otimes
\mathbb{C}^{n|2}$, we should express $L_\mu$ in terms of generalized
Verma modules $M_\nu$, whose character formulae have been known
clearly. Such expressions can help us use some results in classical
simple Lie algebras theory freely.

We also introduce an analogue of Kostant's
$\mathfrak{u}$-cohomology, which can help us not only to describe
the trivial module $L_0$ in terms of $M_\nu$ directly but also to
get much useful information about the tensor module $L_\mu\otimes
\mathbb{C}^{n|2}$. The induction in our argument depends on these
results.

\subsection{}
Our main results are Theorems 5.2, 5.3 and 5.5, which give the
complete character formulae for finite dimensional irreducible
$\mathfrak{osp}(n|2)$-modules.

The paper is organized as follows. In Section 2, we present some
background material on $\mathfrak{osp}(n|2)$. In Section 3, we
introduce an analogue of Kostant's $\mathfrak{u}$-cohomology, by
which we can construct the generalized Kazhdan-Lusztig polynomials
for finite dimensional irreducible $\mathfrak{osp}(n|2)$-modules and
give an explicit expression for the trivial module $L_0$ in terms of
generalized Verma modules. Section 4 is devoted to study the
information of tensor modules $M_\lambda\otimes \mathbb{C}^{n|2}$
and $L_\lambda\otimes \mathbb{C}^{n|2}$. The character formulae for
all finite dimensional irreducible $\mathfrak{osp}(n|2)$-modules are
obtained in Section 5.

\section{Preliminaries}
We explain some basic notions of Lie superalgebras
$\mathfrak{osp}(n|2)$ here and refer to \cite{K1,K2} for more
details. The ground field is the field $\mathbb{C}$ of complex
numbers.
\subsection{Lie superalgebras $\mathfrak{osp}(n|2)$}
Let $V=\mathbb{C}^{n|2}$ be the $\mathbb{Z}_2$-graded vector space
with even subspace $V_{\bar{0}}=\mathbb{C}^{n|0}$ and odd subspace
$V_{\bar{1}}=\mathbb{C}^{0|2}$. The associative algebra $\mbox{End}
\mathbb{C}^{n|2}$ becomes an associative superalgebra if we let
\begin{equation}
\mbox{End}_{\varsigma}\mathbb{C}^{n|2}=\{\xi\in\mbox{End}
\mathbb{C}^{n|2}\mid\xi V_\tau\subset V_{\varsigma+\tau}\},\quad
(\varsigma,\tau\in \mathbb{Z}_2).
\end{equation}
The bracket $[\xi,\eta]=\xi\eta-(-1)^{\deg\xi\deg\eta}\eta\xi$ makes
$\mbox{End} \mathbb{C}^{n|2}$ into a Lie superalgebra, which is
denoted by $\mathfrak{gl}(n|2)$.

Let $F$ be a non-degenerate bilinear form on $V=\mathbb{C}^{n|2}$
such that $V_{\bar{0}}$ and $V_{\bar{1}}$ are orthogonal and the
restriction of $F$ to $V_{\bar{0}}$ is a symmetric and to
$V_{\bar{1}}$ a skew-symmetric form. The ortho-symplectic Lie
superalgebra
$\mathfrak{osp}(n|2)=\mathfrak{osp}(n|2)_{\bar{0}}\oplus\mathfrak{osp}(n|2)_{\bar{1}}$
is defined by
\begin{equation}
\mathfrak{osp}(n|2)_\varsigma=\{\xi\in\mathfrak{gl}(n|2)_\varsigma\mid
F(\xi(x),y)=-(-1)^{\varsigma\deg x}F(x,\xi(y))\},\quad (\varsigma\in
\mathbb{Z}_2).
\end{equation}

In some basis the matrix of $F$ can be written as
\begin{equation}
\left( \begin{array}{cccc} 0 & 1_m &  & \\1_m & 0 &  &\\
&& 0 & 1\\&  & -1 & 0
\end{array}\right), (n=2m);\quad
\left( \begin{array}{ccccc} 0 & 1_m & 0 &  & \\1_m & 0 & 0 & &\\ 0 & 0& 1& & \\
& & & 0 & 1\\& & & -1 & 0
\end{array}\right), (n=2m+1)
\end{equation} from which the elements of $\mathfrak{osp}(n|2)$ can be written as matrices.

Subalgebra $\mathfrak{osp}(n|2)_{\bar{0}}$ consists of matrices of
the form $ \left(\begin{array}{cc} A_{n\times n} & 0  \\ 0&
D_{2\times 2}\end{array} \right); $ while
$\mathfrak{osp}(n|2)_{\bar{1}}$ consists of matrices of the form $
\left(\begin{array}{cc}  0& B_{n\times 2}\\ C_{2\times n}& 0
\end{array} \right). $

It is obvious that, as a Lie algebra,
$\mathfrak{osp}(n|2)_{\bar{0}}\cong\mathfrak{0}(n)\oplus\mathfrak{sl}(2)$.

Let $\mathfrak{h}$ be the subalgebra with all diagonal matrices in
$\mathfrak{osp}(n|2)$, which is called a {\em Cartan subalgebra} of
$\mathfrak{osp}(n|2)$.

The supertrace induces a bilinear form $(\ ,\ ):
\mathfrak{h}^*\times \mathfrak{h}^*\rightarrow \mathbb{C}$ on
$\mathfrak{h}^*$.

\subsection{Distinguished simple root system}
Unlike the classical simple Lie algebras, the dynkin diagram of
$\mathfrak{osp}(n|2)$ depends on the choice of Borel subalgebra. But
there is a unique dynkin diagram with only one odd simple root. Such
a dynkin diagram is called the distinguished Dynkin diagram. The
underlying simple root system is called distinguished simple root
system. We list them below (c.f. \cite{K1,K2}).

\subsubsection{$\mathfrak{osp}(2m|2)$}
\

The set of distinguished simple roots is
\begin{equation}
\Pi=\{\delta-\epsilon_1,
\epsilon_1-\epsilon_{2},\epsilon_2-\epsilon_3,\ldots,\epsilon_{m-1}-\epsilon_m,
\epsilon_{m-1}+\epsilon_m\},
\end{equation} and the set of the positive roots is
\begin{equation}
\Delta^+=\{\delta\pm \epsilon_i, \epsilon_j\pm\epsilon_k,
2\delta\mid 1\leq i\leq m, 1\leq j<k\leq m\}.
\end{equation}
Furthermore, the set of positive even roots and odd roots are
\begin{equation}
\Delta_0^+=\{\epsilon_{i}\pm\epsilon_{j}, 2\delta\mid 1\leq i<j\leq
m\}
\end{equation}
and
\begin{equation}
\Delta_1^+=\{\delta\pm\epsilon_i\mid 1\leq i\leq m\},
\end{equation}
respectively.

The distinguished Dynkin diagram is as follows:

\vspace{0.5cm} \setlength{\unitlength}{3pt}
\begin{picture}(108,6)\put(0,-1){$D(m|1)$:}
\put(21,-1){$\otimes$}\put(17,-5){$\delta-\epsilon_1$}\put(23,0){\line(1,0){12}}
\put(36,0){\circle{2}}\put(32,-5){$\epsilon_1-\epsilon_2$}
\put(42,-0.5){$\cdots$}\put(52,0){\circle{2}}\put(53,0){\line(1,0){12}}
\put(66,0){\circle{2}}\put(56,-5){$\epsilon_{m-2}-\epsilon_{m-1}$}
\put(44,-12){{\tt(Figure 1)}}
\put(66,0){\circle{2}}
\put(67,0){\line(2,1){12}}\put(67,0){\line(2,-1){12}}
\put(80,6){\circle{2}}\put(80,-6){\circle{2}}
\put(82,5){$\epsilon_{m-1}-\epsilon_m$}\put(82,-7){$\epsilon_{m-1}+\epsilon_m$}
\end{picture}
\vspace{1.5cm}

\subsubsection{$\mathfrak{osp}(2m+1|2)$}
\

The set of distinguished simple roots is
\begin{equation}
\Pi=\{\delta-\epsilon_1,
\epsilon_1-\epsilon_{2},\epsilon_2-\epsilon_3,\ldots,\epsilon_{m-1}-\epsilon_m,
\epsilon_m\},
\end{equation} and the set of the positive roots is
\begin{equation}
\Delta^+=\{\delta\pm \epsilon_i, \epsilon_j\pm\epsilon_k, 2\delta,
\delta, \epsilon_i\mid 1\leq i\leq m, 1\leq j<k\leq m\}.
\end{equation} Furthermore, the set of positive even roots and odd roots are
\begin{equation}
\Delta_0^+=\{\epsilon_{j}\pm\epsilon_{k}, 2\delta, \epsilon_i \mid
1\leq i\leq m, 1\leq j<k\leq m\}
\end{equation} and
\begin{equation}
\Delta_1^+=\{\delta\pm\epsilon_i, \delta\mid 1\leq i\leq m\},
\end{equation}
respectively.

The distinguished Dynkin diagram is as follows:

\begin{picture}(108,6)\put(0,-1){$B(m|1)$:}
\put(21,-1){$\otimes$}\put(17,-5){$\delta-\epsilon_1$}\put(23,0){\line(1,0){12}}
\put(36,0){\circle{2}}\put(32,-5){$\epsilon_1-\epsilon_2$}
\put(42,-0.5){$\cdots$}\put(52,0){\circle{2}}\put(53,0){\line(1,0){12}}
\put(66,0){\circle{2}}\put(58,-5){$\epsilon_{m-1}-\epsilon_{m}$}
\put(44,-12){{\tt(Figure 2)}}
\put(66,0){\circle{2}}\put(67,-0.25){\line(1,0){12}}\put(67,0.25){\line(1,0){12}}\put(80,0){\circle{2}}
\put(73,1.75){\line(3,-1){5.5}}\put(73,-1.75){\line(3,1){5.5}}\put(80,-5){$\epsilon_m$}
\end{picture}
\vspace{1.5cm}

In both of the cases $\mathfrak{osp}(2m|2)$ and
$\mathfrak{osp}(2m+1|2)$, the bilinear form $(\ ,\ )$ on
$\mathfrak{h}^*$ satisfies
\begin{equation}
(\delta,\delta)=-1, \quad
(\delta,\epsilon_i)=(\epsilon_i,\delta)=0,\quad
(\epsilon_i,\epsilon_j)=\delta_{ij}, \quad (1\leq i,j\leq m).
\end{equation}

\subsection{$\mathbb{Z}$-grading and parabolic subalgebras}

We shall simplify $\mathfrak{osp}(n|2)$ to $\mathfrak{g}$ throughout
the paper, where $n=2m$ or $2m+1$. The Lie superalgebra
$\mathfrak{g}$ admits a $\mathbb{Z}_2$-consistent
$\mathbb{Z}$-grading
\begin{equation}
\mathfrak{g}=\mathfrak{g}_{-2}\oplus\mathfrak{g}_{-1}\oplus\mathfrak{g}_0\oplus\mathfrak{g}_{1}\oplus\mathfrak{g}_{2},
\end{equation} where
$\mathfrak{g}_0\cong\mathfrak{0}(n)\oplus\mathbb{C}$ is spanned by
$\mathfrak{h}$ and $\mathfrak{g}_{\pm\alpha}
(\alpha\in\Delta_0^+\setminus{2\delta})$; $\mathfrak{g}_{\pm1}$ is
spanned by $\mathfrak{g}_{\pm\alpha} (\alpha\in\Delta_1^+)$; and
$\mathfrak{g}_{\pm2}=\mathfrak{g}_{\pm2\delta}$. There are two
parabolic subalgebras
\begin{equation}
\mathfrak{u}=\mathfrak{g}_{1}\oplus\mathfrak{g}_{2}\quad
\mbox{with}\quad \mathfrak{u}_{\bar{0}}=\mathfrak{g}_{2} \mbox{ and
} \mathfrak{u}_{\bar{1}}=\mathfrak{g}_{1}
\end{equation} and
\begin{equation}
\mathfrak{u}^-=\mathfrak{g}_{-1}\oplus\mathfrak{g}_{-2}\quad
\mbox{with}\quad \mathfrak{u}^-_{\bar{0}}=\mathfrak{g}_{-} \mbox{
and } \mathfrak{u}^-_{\bar{1}}=\mathfrak{g}_{-1}.
\end{equation}

\subsection{Dominant integral weights} An element in $\mathfrak{h}^*$ is called a {\em weight}.
A weight $\lambda\in\mathfrak{h}^*$ will be written in terms of the
$\delta\epsilon$-basis as
\begin{equation}
\lambda=(\lambda_0;\lambda_1,\lambda_2,\ldots,\lambda_m)=\lambda_0\delta+\sum_{i=1}^m
\lambda_i\epsilon_i.
\end{equation}
Define
\begin{equation}
\mbox{ht }\lambda=\sum_{i=0}^m |\lambda_i|,
\end{equation}
and call it the height of $\lambda$.

A weight $\lambda$ is called {\em integral} if
$\lambda_0,\lambda_1,\ldots,\lambda_m\in\mathbb{Z}$.

For any weight $\lambda\in\mathfrak{h}^*$, there exists an
irreducible module $L_\lambda$ with highest weight $\lambda$.

\begin{theorem}{\bf(Kac \cite{K1})} Take any integral weight $\lambda$.

(1) If $\mathfrak{g}=\mathfrak{osp}(2m|2)$, then the irreducible
module $L_{\lambda}$ is finite dimensional if and only if
\begin{equation}
\left\{\begin{array}{l}  \lambda_0,
\lambda_1,\ldots,\lambda_{m-1}\in \mathbb{Z}_{\geq0},\
\lambda_m\in \mathbb{Z},\\
\lambda_1\geq \lambda_2\geq\cdots\geq \lambda_{m-1}\geq|\lambda_m|,\\
\lambda_0\geq t,\quad \mbox{where $t$ is the maximal number such
that $\lambda_t\not=0$.}
\end{array}\right.
\end{equation}

(2) If $\mathfrak{g}=\mathfrak{osp}(2m+1|2)$, then the irreducible
module $L_{\lambda}$ is finite dimensional if and only if
\begin{equation}
\left\{\begin{array}{l}  \lambda_0,
\lambda_1,\ldots,\lambda_{m-1},\lambda_m\in \mathbb{Z}_{\geq0},\\
\lambda_1\geq \lambda_2\geq\cdots\geq \lambda_{m-1}\geq\lambda_m,\\
\lambda_0\geq t,\quad \mbox{where $t$ is the maximal number such
that $\lambda_t\not=0$.}
\end{array}\right.
\end{equation}
\end{theorem}

\vspace{0.3cm} Denote
\begin{equation}
\mathcal {P}=\{\lambda\mid \mbox{$\lambda$ is integral and satisfies
the conditions in Theorem 2.1}\}.
\end{equation}
The weights in $\mathcal{P}$ are called {\em dominant integral}.

It can be checked directly that $\delta$ is a dominant integral
weight. Furthermore, $L_\delta$ is the natural representation of
$\mathfrak{osp}(n|2)$.

\subsection{Generalized Verma modules}
For any $\lambda\in \mathcal{P}$, we denote by $L^{(0)}_\lambda$ the
finite-dimensional irreducible $\mathfrak{g}_0$- module with highest
weight $\lambda$. Extend it to a
$\mathfrak{g}_0\oplus\mathfrak{u}$-module by putting
$\mathfrak{u}L^{(0)}_\lambda=0$. Then the {\em generalized Verma
module} $M_\lambda$ is defined as the induced module
\begin{equation}
M_\lambda=\mbox{Ind}^{\mathfrak{g}}_{\mathfrak{g}_0\oplus\mathfrak{u}}L^{(0)}_\lambda\cong
{U}(\mathfrak{u}^-)\otimes_{\mathbb{C}}L^{(0)}_\lambda.
\end{equation}
It is clear that $L_\lambda$ is the unique irreducible quotient
module of $M_\lambda$.

\subsection{Atypical weights}
Let $\rho_0$ (resp. $\rho_1$) be half the sum of positive even
(resp. odd) roots, and let $\rho=\rho_0-\rho_1$. Then
\begin{equation}
\rho=\left\{\begin{array}{ll}(1-m; m-1,
m-2,\ldots,1,0),&\quad\mbox{if}\quad
\mathfrak{g}=\mathfrak{osp}(2m|2);\\
(\frac{1}{2}-m; m-\frac{1}{2},
m-\frac{3}{2},\ldots,\frac{3}{2},\frac{1}{2}),&\quad\mbox{if}\quad
\mathfrak{g}=\mathfrak{osp}(2m+1|2).\end{array}\right.
\end{equation}

A weight $\lambda\in \mathfrak{h}^*$ is called {\em atypical} if
there is a positive odd root $\alpha\in\Delta_1^+/\{\delta\}$ such
that
\begin{equation}
(\lambda+\rho,\alpha)=0.
\end{equation} Sometimes we also call it {\em $\alpha$-atypical} to emphasize the odd root $\alpha$. Otherwise, we call $\lambda$ {\em
typical}.

It is clear that
\begin{equation}
\lambda \mbox{ is $(\delta\pm\epsilon_i)$-atypical } \Leftrightarrow
\left\{\begin{array}{ll} 1-m+\lambda_0=\pm(m-i+\lambda_i),
&\mbox{if}\quad
\mathfrak{g}=\mathfrak{osp}(2m|2);\\
\frac{1}{2}-m+\lambda_0=\pm(m-i+\frac{1}{2}+\lambda_i),
&\mbox{if}\quad
\mathfrak{g}=\mathfrak{osp}(2m+1|2).\end{array}\right.
\end{equation}

The following lemma is obvious by (2.24).
\begin{lemma}
If $\lambda\in \mathfrak{h}^*$ is $\alpha$-atypical, then
$\lambda\pm\alpha$ are also $\alpha$-atypical.
\end{lemma}

\subsection{Blocks of weights}
For any $\alpha\in\Delta_1^+/\{\delta\}$, define a map $t_{\alpha}:
\mathfrak{h}^*\rightarrow\mathfrak{h}^*$ by
\begin{equation}
t_{\alpha}(\lambda)=\left\{\begin{array}{ll}\lambda+\alpha,&
\mbox{if $\lambda$ is $\alpha$-atypical};\\
\lambda, & \mbox{otherwise}
\end{array}\right.
\end{equation} and define $t_\alpha=t^{-1}_{-\alpha}$ for any $\alpha\in\Delta_1^-/\{\delta\}$

Let $\mathcal{W}$ be the Weyl group of $\mathfrak{g}$ (i.e. the Weyl
group of Lie algebra $\mathfrak{g}_{\bar{0}}$). For any
$w\in\mathcal{W}$, define a map $t_w:
\mathfrak{h}^*\rightarrow\mathfrak{h}^*$ by
\begin{equation}
t_w(\lambda)=w(\lambda+\rho)-\rho.
\end{equation}

Denote
\begin{equation}
T=\{t_\alpha, t_w\mid \alpha\in\Delta_1/\{\delta\},
w\in\mathcal{W}\}.
\end{equation}

We say two weights $\lambda, \mu\in\mathfrak{h}^*$ are \emph{in the
same block} (denote by $\lambda\sim\mu$) if there exist $r_1,
r_2,\ldots,r_p\in T$ such that $\lambda=r_1r_2\cdots r_p\mu$.

\subsection{Partial orders}
Define a partial order ``$\succ$'' in $\mathcal {P}$ by
\begin{equation}
\lambda\succ\mu\quad\Leftrightarrow\quad \lambda\sim\mu\mbox{ and }
\mbox{ht }\lambda>\mbox{ht }\mu.
\end{equation}

In this paper, we also always use the natural partial order ``$>$''
in $\mathfrak{h}^*$ by
\begin{equation}
\lambda>\mu \quad\Leftrightarrow\quad \lambda-\mu \mbox{ is a
$\mathbb{Z}_{\geq 0}$-linear sum of positive roots}.
\end{equation}

\subsection{Character formulae for typical weights}
It is obvious that any typical dominant integral weight
$\lambda\in\mathcal {P}$ satisfies that $\lambda_0\geq m$.

Denote
\begin{equation}
\mathcal {P}_{\mathfrak{o}(n)}=\{\mbox{dominant integral weights of
}\mathfrak{o}(n)\}.
\end{equation} For any $\lambda=\sum_{i=1}^m \lambda_i\epsilon_i\in\mathcal {P}_{\mathfrak{o}(n)}$, denote by
$\mathcal {L}_\lambda$ the irreducible $\mathfrak{o}(n)$-module with
weight $\lambda$.

For any $\lambda=\lambda_0\delta+\sum_{i=1}^m \lambda_i\epsilon_i\in
\mathcal{P}$, write
\begin{equation}
\lambda_{\mathfrak{o}(n)}=\sum_{i=1}^m
\lambda_i\epsilon_i\in\mathcal {P}_{\mathfrak{o}(n)}
\end{equation} and
\begin{equation}
\lambda^\delta=(n-2-\lambda_0)\delta+\sum_{i=1}^m
\lambda_i\epsilon_i.
\end{equation}
It is clear that $\lambda\sim\lambda^\delta$.

\begin{lemma}
For any $\lambda\in\mathcal {P}$ with $\lambda_0\geq m$,
\begin{equation}
\mbox{ch }M_\lambda-\mbox{ch
}M_{\lambda^\delta}=\frac{\prod_{\alpha\in\Delta_1^+}(e^{\frac{\alpha}{2}}+e^{-\frac{\alpha}{2}})}
{\prod_{\alpha\in\Delta_0^+}(e^{\frac{\alpha}{2}}-e^{-\frac{\alpha}{2}})}\sum_{w\in\mathcal
{W}}\varepsilon(w)e^{w(\lambda+\rho)}.
\end{equation}
\end{lemma}
\begin{proof}
For any $\lambda\in\mathcal {P}$, it is obvious that
$\lambda_{\mathfrak{o}(n)}\in\mathcal {P}_{{\mathfrak{o}(n)}}$

Notice that by definition (2.25),
\begin{equation}
\mbox{ch
}M_\nu=\frac{\prod_{\alpha\in\Delta_1^+}(1+e^{-\alpha})}{1-e^{-2\delta}}\mbox{ch
}L_\nu^{(0)}.
\end{equation}
Hence
\begin{eqnarray*}
\mbox{ch }M_\lambda-\mbox{ch
}M_{\lambda^\delta}&=&\frac{\prod_{\alpha\in\Delta_1^+}(1+e^{-\alpha})}{1-e^{-2\delta}}
(e^{\lambda_0\delta}-e^{(n-2-\lambda_0)\delta})\mbox{ch }\mathcal
{L}_{\lambda_{\mathfrak{o}(n)}}\\
&=&\frac{\prod_{\alpha\in\Delta_1^+}(1+e^{-\alpha})}{1-e^{-2\delta}}
(e^{\lambda_0\delta}-e^{(n-2-\lambda_0)\delta})\frac{\sum_{w\in
\mathcal
{W}_{\mathfrak{o}(n)}}\varepsilon(w)e^{w(\lambda_{\mathfrak{o}(n)}+\rho_{\mathfrak{o}(n)})-\rho_{\mathfrak{o}(n)}}}
{\prod_{\alpha\in\Delta_{\mathfrak{o}(n)}^+}(1-e^{-\alpha})}\\
&=&\frac{\prod_{\alpha\in\Delta_1^+}(1+e^{-\alpha})}{\prod_{\alpha\in\Delta_0^+}(1-e^{-\alpha})}
\sum_{w\in\mathcal {W}}\varepsilon(w)e^{w(\lambda+\rho)-\rho}\\
&=&\frac{\prod_{\alpha\in\Delta_1^+}(1+e^{-\alpha})}{\prod_{\alpha\in\Delta_0^+}(1-e^{-\alpha})}
\frac{\prod_{\alpha\in\Delta_1^+}e^{\frac{\alpha}{2}}}{\prod_{\alpha\in\Delta_0^+}e^{\frac{\alpha}{2}}}
\sum_{w\in\mathcal {W}}\varepsilon(w)e^{w(\lambda+\rho)}\\&=&
\frac{\prod_{\alpha\in\Delta_1^+}(e^{\frac{\alpha}{2}}+e^{-\frac{\alpha}{2}})}
{\prod_{\alpha\in\Delta_0^+}(e^{\frac{\alpha}{2}}-e^{-\frac{\alpha}{2}})}\sum_{w\in\mathcal
{W}}\varepsilon(w)e^{w(\lambda+\rho)},
\end{eqnarray*} where $\mathcal
{W}_{\mathfrak{o}(n)}$ is the Weyl group of $\mathfrak{o}(n)$,
$\Delta_{\mathfrak{o}(n)}^+=\Delta_0^+\setminus\{2\delta\}$ is the
positive roots of $\mathfrak{o}(n)$, and
$\rho_{\mathfrak{o}(n)}=\rho+(\frac{n}{2}-1)\delta$.
\end{proof}

At the same time, it is well known that
\begin{theorem}{\bf(Kac \cite{K3})}
Weight $\lambda\in\mathcal {P}$ is typical if and only if
\begin{equation}
\mbox{ch }
L_\lambda=\frac{\prod_{\alpha\in\Delta_1^+}(e^{\frac{\alpha}{2}}+e^{-\frac{\alpha}{2}})}
{\prod_{\alpha\in\Delta_0^+}(e^{\frac{\alpha}{2}}-e^{-\frac{\alpha}{2}})}\sum_{w\in\mathcal
{W}}\varepsilon(w)e^{w(\lambda+\rho)}.
\end{equation}
\end{theorem}

Thanks to Lemma 2.4 and Theorem 2.5, we can obtain the following
corollary immediately.

\begin{corollary} For any $\lambda\in\mathcal {P}$,
\begin{equation}
\mbox{ch } L_\lambda=\mbox{ch } M_\lambda-\mbox{ch }
M_{\lambda^\delta}
\end{equation} if and only if $\lambda$ is
typical.
\end{corollary}

\section{Cohomology and character}

\subsection{Definition of cohomology}
The space of $q$-dimensional cochains of the Lie superalgebra
$\mathfrak{u}=\mathfrak{u}_{\bar{0}}\otimes\mathfrak{u}_{\bar{1}}$
with coefficients in the module $L_\lambda$ is given by
\begin{equation}
C^q(\mathfrak{u};
L_\lambda)=\bigoplus_{q_0+q_1=q}\mbox{Hom}(\wedge^{q_0}\mathfrak{u}_{\bar{0}}\otimes
S^{q_1}\mathfrak{u}_{\bar{1}}, L_\lambda).
\end{equation}
The differential $d: C^q(\mathfrak{u}; L_\lambda)\rightarrow
C^{q+1}(\mathfrak{u}; L_\lambda)$ is defined by
\begin{equation}
\begin{array}{c}
d c(\xi_1,\ldots,\xi_{q_0},\eta_1,\ldots,\eta_{q_1})\\=\sum_{1\leq s
\leq t\leq
q_0}(-1)^{s+t-1}c([\xi_s,\xi_t],\xi_1,\ldots,\widehat{\xi_s},\ldots,\widehat{\xi_t},\ldots,
\xi_{q_0},\eta_1,\ldots,\eta_{q_1})
\\
+\sum_{s=1}^{q_0}\sum_{t=1}^{q_1}(-1)^{s-1}c(\xi_1,\ldots,\widehat{\xi_s},\ldots,\xi_{q_0},
[\xi_s,\eta_t],\eta_1,\ldots,\widehat{\eta_t},\ldots,\eta_{q_1})\\
+\sum_{1\leq s\leq t\leq
q_1}c([\eta_s,\eta_t],\xi_1,\ldots,\xi_{q_0},
\eta_1,\ldots,\widehat{\eta_s},\ldots,\widehat{\eta_t},\ldots,\eta_{q_1})\\
+\sum_{s=1}^{q_0}(-1)^s\xi_s c(\xi_1,\ldots,\widehat{\xi_s},\ldots,\xi_{q_0},\eta_1,\ldots,\eta_{q_1})\\
+(-1)^{q_0-1}\sum_{s=1}^{q_1}\eta_s
c(\xi_1,\ldots,\xi_{q_0},\eta_1,\ldots,\widehat{\eta_s},\ldots,\eta_{q_1})
\end{array}\end{equation}
where $c\in C^{q}(\mathfrak{u}; L_\lambda)$,
$\xi_1,\ldots,\xi_{q_0}\in\mathfrak{u}_{\bar{0}}$,
$\eta_1,\ldots,\eta_{q_1}\in\mathfrak{u}_{\bar{1}}$.

The cohomology of $\mathfrak{u}$ with coefficients in the module
$L_\lambda$ is the cohomology groups of the complex
$C=(\{C^q(\mathfrak{u}; L_\lambda)\},d)$, and it is denoted by
$H^q(\mathfrak{u}, L_\lambda)$.

\subsection{Euler-Poincar$\acute{\mbox{e}}$ principle}
The Euler-Poincar$\acute{\mbox{e}}$ principle says that, for each
weight $\nu$,
\begin{equation}
\sum_{i=0}^\infty(-1)^i\dim H^i(\mathfrak{u},
L_\lambda)_\nu=\sum_{i=0}^\infty(-1)^i\dim C^i(\mathfrak{u};
L_\lambda)_\nu,
\end{equation}
where $H^i(\mathfrak{u}, L_\lambda)_\nu$ and $C^i(\mathfrak{u};
L_\lambda)_\nu$ are the weight-$\nu$ subspace of $H^i(\mathfrak{u},
L_\lambda)$ and $C^i(\mathfrak{u}; L_\lambda)$, respectively. Thus
taking the formal characters, we have
\begin{equation}
\sum_{i=0}^\infty(-1)^i\mbox{ch } H^i(\mathfrak{u},
L_\lambda)=\sum_{i=0}^\infty(-1)^i\mbox{ch } C^i(\mathfrak{u};
L_\lambda).
\end{equation}

\subsection{Kazhdan-Lusztig polynomials}
Recall that
$\mathfrak{u}_{\bar{0}}=\mathfrak{g}_2=\mathfrak{g}_{2\delta}$ and
$\mathfrak{u}_{\bar{1}}=\mathfrak{g}_1=\bigoplus_{\alpha\in\Delta_1^+}\mathfrak{g}_\alpha$.
So
\begin{equation}
C^i(\mathfrak{u}; L_\lambda)\simeq (\mathfrak{g}_{2}^*\otimes
S^{i-1}(\mathfrak{g}_1^*)\oplus S^{i}(\mathfrak{g}_1^*))\otimes
L_\lambda.
\end{equation}
Hence we can calculate that
\begin{eqnarray}
\sum_{i=0}^\infty(-1)^i\mbox{ch } C^i(\mathfrak{u};
L_\lambda)=\frac{(1-e^{-2\delta})\mbox{ch
}L_\lambda}{\prod_{\alpha\in\Delta_1^+}(1+e^{-\alpha})}.
\end{eqnarray}

Combining (3.4) and (3.6), we get
\begin{equation}
\mbox{ch } L_\lambda=\sum_{i=0}^\infty(-1)^i\mbox{ch }
H^i(\mathfrak{u},
L_\lambda)\frac{\prod_{\alpha\in\Delta_1^+}(1+e^{-\alpha})}{1-e^{-2\delta}}.
\end{equation}

Let $[H^i(\mathfrak{u},L_\lambda): L_\nu^{(0)}]$ denote the
multiplicity of $L_\nu^{(0)}$ in the cohomology group regarded as a
$\mathfrak{g}_{0}$ module. An easy inspection of the complex
$C=(\{C^q(\mathfrak{u}; L_\lambda)\},d)$ reveals that
\begin{equation}
[H^i(\mathfrak{u},L_\lambda):
L_\lambda^{(0)}]=\left\{\begin{array}{ll}1,& (i=0);\\
0,&(i>0),\end{array}\right.
\end{equation}
\begin{equation}
[H^i(\mathfrak{u},L_\lambda): L_\mu^{(0)}]=0, \quad(\forall
\mu>\lambda).
\end{equation}

In terms of these multiplicities, $\mbox{ch } H^i(\mathfrak{u},
L_\lambda)$ can be expressed as
\begin{equation}
\mbox{ch } H^i(\mathfrak{u}, L_\lambda)=\sum_\nu[H^i(\mathfrak{u},
L_\lambda): L_\nu^{(0)}]\mbox{ch }L_\nu^{(0)}.
\end{equation}
Notice that by definition (2.21),
\begin{equation}
\mbox{ch
}M_\nu=\frac{\prod_{\alpha\in\Delta_1^+}(1+e^{-\alpha})}{1-e^{-2\delta}}\mbox{ch
}L_\nu^{(0)}.
\end{equation}

By (3.7), (3.8) and (3.9), we obtain the following lemma.
\begin{lemma}
The formal character of $L_\lambda$ can be expressed as
\begin{equation}
\mbox{ch }
L_\lambda=\sum_\nu\sum_{i=0}^\infty(-1)^i[H^i(\mathfrak{u},
L_\lambda): L_\nu^{(0)}]\mbox{ch } M_\nu.
\end{equation}
\end{lemma}

\begin{remark}
In the above expression, any coefficient of $\mbox{ch } M_\nu$
should be an integral number. Moreover, $[H^i(\mathfrak{u},
L_\lambda): L_\nu^{(0)}]=0$ if $\lambda\not\sim\nu$.
\end{remark}

\begin{remark}
Thanks to (3.8) and (3.9),
\begin{equation}
\mbox{ch } L_\lambda=\mbox{ch }
M_\lambda+\sum_{\nu<\lambda}\sum_{i=0}^\infty(-1)^i[H^i(\mathfrak{u},
L_\lambda): L_\nu^{(0)}]\mbox{ch } M_\nu.
\end{equation} Notice that the coefficient of $\mbox{ch }
M_\lambda$ is $1$, which is very important in our arguments below.
\end{remark}

\subsection{$\mathfrak{u}$-cohomology with trivial coefficients}
Below we shall take $\lambda=0$ and compute the cohomology
$H(\mathfrak{u}, L_0)$.

Let $\{e_\alpha\in\mathfrak{g}_\alpha\mid \alpha\in\Delta\}$ be a
Chevalley basis of $\mathfrak{g}$.

Now (3.1) and (3.2) can be simplified to
\begin{equation}
C^q(\mathfrak{u})=\mbox{Hom}(S^q\mathfrak{u}_{\bar{1}},\mathbb{C})\oplus\mbox{Hom}(\mathfrak{g}_{2\delta}\otimes
S^{q-1}\mathfrak{u}_{\bar{1}},\mathbb{C})
\end{equation}
and
\begin{equation}
\left\{\begin{array}{l}d
c(\eta_1,\ldots,\eta_{q})=\sum_{[\eta_i,\eta_j]=k e_{2\delta},
i<j}kc(e_{2\delta},\eta_1,\ldots,\widehat{\eta_i},\ldots,\widehat{\eta_j},\ldots,\eta_{q}),\\
d c(e_{2\delta},\eta_1,\ldots,\eta_{q-1})=0,
\end{array}\right.
\end{equation}
where $\eta_1,\ldots,\eta_{q}\in\mathfrak{u}_{\bar{1}}$.

Therefore the space of $q$-dimensional cochains is
\begin{equation}
Z^q(\mathfrak{u})=\mathbb{C}[e_{\alpha}^*\mid \alpha\in\Delta_1^+]_q
\end{equation}
and the space of $q$-dimensional coboundaries is
\begin{equation}
B^q(\mathfrak{u})=\left\{\begin{array}{ll}
(\sum_{i=1}^me_{\delta-\epsilon_i}^*e_{\delta+\epsilon_i}^*)\mathbb{C}[e_{\alpha}^*\mid
\alpha\in\Delta_1^+]_{q-2}, &\mbox{if}\quad
\mathfrak{g}=\mathfrak{osp}(2m|2);\\
(e_{\delta}^*e_{\delta}^*+\sum_{i=1}^me_{\delta-\epsilon_i}^*e_{\delta+\epsilon_i}^*)\mathbb{C}[e_{\alpha}^*\mid
\alpha\in\Delta_1^+]_{q-2}, &\mbox{if}\quad
\mathfrak{g}=\mathfrak{osp}(2m+1|2).
\end{array}\right.
\end{equation}
where $e_\alpha^*$ is the dual basis of $e_\alpha$ in
$\mathfrak{g}^*$.

Hence by the representation theory of orthogonal Lie algebra
$\mathfrak{0}(n)$, we have the following result.
\begin{proposition}
As a $\mathfrak{g}_{0}$-module,
\begin{equation}
H^q(\mathfrak{u})\cong L^{(0)}_{q\epsilon_1-q\delta}.
\end{equation}
\end{proposition}

Applying Lemma 3.1 to the above formula, we get
\begin{corollary}The formal character of the trivial module $L_0$ can be expressed as
\begin{equation}
\mbox{ch } L_0=\sum_{i=0}^\infty(-1)^i\mbox{ch
}M_{i\epsilon_1-i\delta}.
\end{equation}
\end{corollary}

\subsection{Convention} From now on, we will always simplify the character
$\mbox{ch } V$ to $V$. It cannot confuse us by context. For example,
we can rewrite (3.19) as
\begin{equation}
L_0=\sum_{i=0}^\infty(-1)^iM_{i\epsilon_1-i\delta}.
\end{equation}

\section{Tensor modules}

\subsection{Decomposition of $\mathfrak{o}(n)$-module $\mathcal {L}_\lambda\otimes \mathcal
{L}_{\epsilon_1}$} Recall (2.30) that
\begin{equation}
\mathcal {P}_{\mathfrak{o}(n)}=\{\mbox{dominant integral weights of
}\mathfrak{o}(n)\}
\end{equation} and that
$\mathcal {L}_\lambda$ is the irreducible $\mathfrak{o}(n)$-module
with highest weight $\lambda\in\mathcal {P}_{\mathfrak{o}(n)}$.

It is well known that $\mathcal {L}_{\epsilon_1}$ is the natural
representation of $\mathfrak{o}(n)$. The following lemma is standard
in the theory of classical simple Lie algebras.

\begin{lemma} (1)
For any $\lambda=\sum_{i=1}^m \lambda_i\epsilon_i\in\mathcal
{P}_{\mathfrak{o}(2m)}$ satisfying
\begin{equation}
\left\{\begin{array}{l}\lambda_1,\lambda_2,\ldots,\lambda_{m-1}\in\mathbb{Z}_{\geq0},
\lambda_m\in\mathbb{Z};\\
\lambda_1\geq
\lambda_2\geq\cdots\geq\lambda_{m-1}\geq|\lambda_m|,\end{array}\right.
\end{equation} it should be that
\begin{equation}
\mathcal {L}_\lambda\otimes \mathcal
{L}_{\epsilon_1}=\bigoplus_{\mu\in\mathcal
{P}_{\mathfrak{o}(2m)}\cap \{\lambda\pm\epsilon_i\mid 1\leq i\leq
m\}}\mathcal {L}_\mu.
\end{equation}
(2) For any $\lambda=\sum_{i=1}^m \lambda_i\epsilon_i\in\mathcal
{P}_{\mathfrak{o}(2m+1)}$ satisfying
\begin{equation}
\left\{\begin{array}{l}\lambda_1,\lambda_2,\ldots,\lambda_{m-1},\lambda_m\in\mathbb{Z}_{\geq0};\\
\lambda_1\geq
\lambda_2\geq\cdots\geq\lambda_{m-1}\geq\lambda_m,\end{array}\right.
\end{equation} it should be that
\begin{equation}
\mathcal {L}_\lambda\otimes \mathcal
{L}_{\epsilon_1}=\left\{\begin{array}{ll} \bigoplus_{\mu\in\mathcal
{P}_{\mathfrak{o}(2m)}\cap \{\lambda\pm\epsilon_i\mid 1\leq i\leq
m\}}\mathcal {L}_\mu, & \mbox{if}\quad \lambda_m=0;\\
\mathcal {L}_\lambda\oplus\bigoplus_{\mu\in\mathcal
{P}_{\mathfrak{o}(2m)}\cap \{\lambda\pm\epsilon_i\mid 1\leq i\leq
m\}}\mathcal {L}_\mu, & \mbox{if}\quad
\lambda_m\neq0.\end{array}\right.
\end{equation}
\end{lemma}

For convenience, we denote
\begin{equation}
\Upsilon_\lambda=\{\mu\mid \mathcal {L}_\mu \mbox{ is a summand of
}\mathcal {L}_\lambda\otimes \mathcal {L}_{\epsilon_1}\}.
\end{equation}

\subsection{Character of $M_\lambda\otimes
L_\delta$} Consider the natural representation
$L_{\delta}\simeq\mathbb{C}^{n|2}$ of $\mathfrak{osp}(n|2)$. It is
known that the set of all weights of $L_\delta$ is $\{\pm\delta,
\pm\epsilon_1,\ldots,\pm\epsilon_m\}$ if $n=2m$ and $\{0,\pm\delta,
\pm\epsilon_1,\ldots,\pm\epsilon_m\}$ if $n=2m+1$. The dual module
$L_\delta^*\simeq L_\delta$. Furthermore, as a
$\mathfrak{g}_0$-module, $L_\delta=L^{(0)}_\delta\oplus
L^{(0)}_{\epsilon_1}$. Thus by (2.21) and (4.6), we have that, as a
$\mathfrak{g}_0$-module,
\begin{eqnarray}
\quad\quad\quad M_\lambda\otimes
L_\delta&=&({U}(\mathfrak{u}^-)\otimes_{\mathbb{C}}
L^{(0)}_\lambda)\otimes
L_\delta\\\nonumber&\cong&{U}(\mathfrak{u}^-)\otimes_{\mathbb{C}}(L^{(0)}_\lambda\otimes
L_\delta)\\\nonumber
&=&{U}(\mathfrak{u}^-)\otimes_{\mathbb{C}}(L^{(0)}_{\lambda+\delta}\oplus
L^{(0)}_{\lambda-\delta}\oplus
\bigoplus_{\mu-\lambda_0\delta\in\Upsilon_{\lambda-\lambda_0\delta}}L^{(0)}_\mu)\\\nonumber&=&
({U}(\mathfrak{u}^-)\otimes_{\mathbb{C}}L^{(0)}_{\lambda+\delta})\oplus
({U}(\mathfrak{u}^-)\otimes_{\mathbb{C}}L^{(0)}_{\lambda-\delta})
\oplus\bigoplus_{\mu}({U}(\mathfrak{u}^-)\otimes_{\mathbb{C}}L^{(0)}_{\mu})\\\nonumber
&=& M_{\lambda+\delta}\oplus
M_{\lambda-\delta}\oplus\bigoplus_{\mu-\lambda_0\delta\in\Upsilon_{\lambda-\lambda_0\delta}}
M_{\mu}.
\end{eqnarray}

\begin{remark}
In the above calculation, we use the so-called ``Tensor Identity''
in the two module isomorphisms, whose proof is as the same as
Proposition 1.7 in \cite{GL}.
\end{remark}

Hence
\begin{equation}
\mbox{ch }M_\lambda\otimes L_\delta=\mbox{ch
}M_{\lambda+\delta}+\mbox{ch }
M_{\lambda-\delta}+\sum_{\mu-\lambda_0\delta\in\Upsilon_{\lambda-\lambda_0\delta}}
\mbox{ch }M_{\mu}.
\end{equation}

Recall the short notation introduced in 3.5. We can rewrite (4.8) by
Lemma 4.1 to obtain the following lemma.

\begin{lemma} 1). If $\mathfrak{g}=\mathfrak{osp}(2m|2)$, then
for any generalized Verma module $M_\lambda$,
\begin{equation}
M_\lambda\otimes
L_\delta=M_{\lambda+\delta}+M_{\lambda-\delta}+\sum_{\mu\in\{\lambda\pm\epsilon_i\mid
1\leq i\leq m\}, \mu-\lambda_0\delta \in \mathcal
{P}_{\mathfrak{0}(2m)}}M_{\mu}.
\end{equation}
2). If $\mathfrak{g}=\mathfrak{osp}(2m+1|2)$, then for any
generalized Verma module $M_\lambda$,
\begin{equation}
M_\lambda\otimes
L_\delta=\left\{\begin{array}{ll}M_{\lambda+\delta}+M_{\lambda-\delta}+\sum_{\mu\in\{\lambda\pm\epsilon_i\mid
1\leq i\leq m\}, \mu-\lambda_0\delta \in \mathcal
{P}_{\mathfrak{0}(2m+1)}}M_{\mu},& (\lambda_m=0);\\
M_{\lambda+\delta}+M_{\lambda-\delta}+\sum_{\mu\in\{\lambda,\lambda\pm\epsilon_i\mid
1\leq i\leq m\}, \mu-\lambda_0\delta \in \mathcal
{P}_{\mathfrak{0}(2m+1)}}M_{\mu},&
(\lambda_m\neq0).\end{array}\right.
\end{equation}
\end{lemma}

\subsection{Weights set $\mathcal{P}_\lambda$}
Denote
\begin{equation}
\mathcal{P}_\lambda=\mathcal{P}_{\lambda^+}\cup\mathcal{P}_{\lambda^-}
\end{equation}
where
\begin{equation}
\mathcal{P}_{\lambda^+}=\{\lambda+\delta,\lambda+\epsilon_1,\ldots,\lambda+\epsilon_{m-1},\lambda+\epsilon_m
(\mbox{if } \lambda_m\geq 0),\lambda-\epsilon_m (\mbox{if }
\lambda_m\leq 0)\}\cap \mathcal{P},
\end{equation}
\begin{equation}
\mathcal{P}_{\lambda^-}=\{\lambda-\delta,\lambda-\epsilon_1,\ldots,\lambda-\epsilon_{m-1},\lambda+\epsilon_m
(\mbox{if } \lambda_m<0),\lambda-\epsilon_m (\mbox{if }
\lambda_m>0)\}\cap \mathcal{P}
\end{equation} if $\mathfrak{g}=\mathfrak{osp}(2m|2)$, and
\begin{equation}
\mathcal{P}_{\lambda^+}=\left\{\begin{array}{ll}
\{\lambda+\delta,\lambda+\epsilon_1,\ldots,\lambda+\epsilon_{m-1},\lambda+\epsilon_m\}\cap
\mathcal{P}, & (\lambda_m=0);\\
\{\lambda,\lambda+\delta,\lambda+\epsilon_1,\ldots,\lambda+\epsilon_{m-1},\lambda+\epsilon_m\}\cap
\mathcal{P}, & (\lambda_m\neq0).\end{array}\right.,
\end{equation}
\begin{equation}
\mathcal{P}_{\lambda^-}=\left\{\begin{array}{ll}
\{\lambda-\delta,\lambda-\epsilon_1,\ldots,\lambda-\epsilon_{m-1},\lambda-\epsilon_m\}\cap
\mathcal{P}, & (\lambda_m=0);\\
\{\lambda,\lambda-\delta,\lambda-\epsilon_1,\ldots,\lambda-\epsilon_{m-1},\lambda-\epsilon_m\}\cap
\mathcal{P}, & (\lambda_m\neq0).\end{array}\right.
\end{equation}
if $\mathfrak{g}=\mathfrak{osp}(2m+1|2)$.

\begin{lemma}
Suppose $\lambda\in \mathcal{P}$ is an atypical weight. Take any
$\mu,\nu\in \mathcal{P}_\lambda$ with $\mu\neq\nu$.

1). For $\mathfrak{g}=\mathfrak{osp}(2m|2)$, $\mu\sim\nu$ if and
only if $\lambda_0=m-1, \lambda_{m-1}=\lambda_m=0$ and
$\mu,\nu=\lambda\pm\delta$.

2). For $\mathfrak{g}=\mathfrak{osp}(2m+1|2)$, it must be
$\mu\not\sim\nu$.
\end{lemma}
\begin{proof} We prove the statement only for the case of
$\mathfrak{g}=\mathfrak{osp}(2m|2)$ here. The proof for the case of
$\mathfrak{g}=\mathfrak{osp}(2m+1|2)$ is similar.

When $\lambda_0=m-1$, it must be that $\lambda_m=0$ and $\lambda$ is
$(\delta\pm\epsilon_m)$-atypical by (2.26). We can check that only
$\lambda+\delta\sim\lambda-\delta$ in
$\{\lambda+\xi\mid\xi=\pm\delta,\pm\epsilon_1,\ldots,\pm\epsilon_m\}$.
And if $\lambda-\delta\in\mathcal {P}_\lambda$, then
$\lambda_{m-1}=0$.

Assume $\lambda_0\neq m-1$.

If $\lambda$ is $(\delta-\epsilon_k)$-atypical, then $0\neq
1-m+\lambda_0=-m+k-\lambda_k\leq0$. Suppose that there exist
$\mu,\nu\in \mathcal{P}_\lambda$ such that $\mu\sim\nu$. Then it
must be that $\mu,\nu$ are both $(\delta-\epsilon_{k'})$-atypical
for a certain $k'$ since $1-m+\lambda_0<0$. That is,
$\mu=\lambda+\delta, \nu=\lambda+\epsilon_{k'}$ or
$\mu=\lambda-\delta, \nu=\lambda-\epsilon_{k'}$. For the first case,
$2-m+\lambda_0=-m+k'-\lambda_{k'} \ \Rightarrow
k'-\lambda_{k'}=k-\lambda_k+1\ \Rightarrow k'=k+1,
\lambda_{k'}=\lambda_k$. It is impossible since
$\nu=\lambda+\epsilon_{k'}\in \mathcal{P}$. Similarly, it is
impossible for the second case.

If $\lambda$ is $(\delta+\epsilon_k)$-atypical, it is also the same
to show there is no $\mu,\nu\in\mathcal {P}_\lambda$ such that
$\mu\neq\nu$ and $\mu\sim\nu$.
\end{proof}

\subsection{Submodules and quotient modules of $L_\lambda\otimes
L_\delta$} The results stated in this subsection hold for both the
cases of $\mathfrak{g}=\mathfrak{osp}(2m|2)$ and
$\mathfrak{g}=\mathfrak{osp}(2m+1|2)$. Since the proofs for these
two cases are similar to each other, we shall always consider only
the case of $\mathfrak{osp}(2m|2)$ in the arguments.

\begin{lemma} For any $\mu\in\mathcal
{P}_{\lambda^+}$,
\begin{equation}
[L_\lambda\otimes L_\delta: L_\mu]=1.
\end{equation} Particularly, if $\mu$ is typical, then $L_\mu$ is a direct
summand in $L_\lambda\otimes L_\delta$, hence
\begin{equation}
\dim\mbox{Hom}_{\mathfrak{g}}(L_\lambda\otimes L_\delta,
L_\mu)=\dim\mbox{Hom}_{\mathfrak{g}}( L_\mu,L_\lambda\otimes
L_\delta)=1.
\end{equation}
\end{lemma}
\begin{proof}
We prove the case of $\mathfrak{g}=\mathfrak{osp}(2m|2)$ only.

For any $\lambda\in\mathcal {P}$ with $\lambda^\delta\in\mathcal
{P}$, we shall get the statement by a direct calculation in next
section (see Remark 5.4).

If $\lambda\in\mathcal {P}$ with $\lambda^\delta\not\in\mathcal
{P}$. Without loss of generality, we assume $\lambda_m\geq0$. Now
$\lambda$ is either typical or $(\delta+\epsilon_i)$-atypical
($1\leq i\leq m$). In both cases, we have $\lambda\not\sim
\lambda-\xi$ for any
$\xi\in\{\delta-\epsilon_j,\epsilon_k-\epsilon_l\mid1\leq j\leq
m,1\leq k<l\leq m\}$. Thus for any $\nu\sim\lambda$ with
$\nu<\lambda$, it should be that
\begin{equation}
\{\nu\pm\delta,\nu\pm\epsilon_j\mid 1\leq j\leq m\}\cap\mathcal
{P}_{\lambda^+}=\emptyset.
\end{equation}
Therefore if we multiply $L_\delta$ on the both sides of (3.13) and
calculate the right side by Lemma 4.3, we can obtain that the
coefficients of $M_\mu$ ($\forall \mu\in\mathcal {P}_{\lambda^+}$)
are all exactly $1$. Hence by Remark 3.3 again we get
\begin{equation}
[L_\lambda\otimes L_\delta: L_\mu]=1 \quad\mbox{for any
$\mu\in\mathcal {P}_{\lambda^+}$.}
\end{equation}
\end{proof}

\begin{lemma}
For any $\lambda\in \mathcal{P}$, all irreducible submodules and
quotient modules of $L_\lambda\otimes L_\delta$ have to be with form
$L_\mu$ ($\mu\in \mathcal{P}_\lambda$).
\end{lemma}
\begin{proof} For any $\lambda\in\mathcal {P}$ with $\lambda^\delta\in\mathcal
{P}$, we shall obtain the statement by a direct calculation in next
section (see Remark 5.4).

Suppose $\lambda\in\mathcal {P}$ with $\lambda^\delta\not\in\mathcal
{P}$. Without loss of generality, we assume $\lambda_m\geq0$. In
this case, $\lambda$ is either typical or
$(\delta+\epsilon_i)$-atypical ($1\leq i\leq m$).

For any $\nu\sim\lambda$ with $\nu<\lambda$, there exists no weight
$\mu\in\{\nu\pm\delta,\nu\pm\epsilon_i\mid 1\leq i\leq m\}$ such
that $\mu\in\mathcal {P}$ and $\mbox{ht } \mu>\mbox{ht }\lambda$.
Thus if we multiply $L_\delta$ on the right side of (3.13), then the
coefficients of $M_\mu$, where $\mbox{ht }\mu>\mbox{ht }\lambda$, is
nonzero if and only if $\mu\in\mathcal {P}_{\lambda^+}$. Therefore
by Remark 3.3
\begin{equation}
[L_\lambda\otimes L_\delta : L_\mu]=0 \quad \mbox{for any
$\mu\in\mathcal {P}\setminus \mathcal {P}_\lambda$ with $\mbox{ht
}\mu>\mbox{ht }\lambda$.}
\end{equation}

Take any irreducible submodule or quotient module $L_\mu$ of
$L_\lambda\otimes L_\delta$.

Suppose $\mbox{ht }\mu>\mbox{ht }\lambda$. It must be that
$\mu\in\mathcal {P}_\lambda$ because of (4.20).

Suppose $\mbox{ht }\mu<\mbox{ht }\lambda$. Since
\begin{equation}
\mbox{Hom}_\mathfrak{g}(L_\lambda\otimes L_\delta,
L_\mu)\simeq\mbox{Hom}_\mathfrak{g}(L_\lambda, L_\mu\otimes
L_\delta)
\end{equation} and
\begin{equation}
\mbox{Hom}_\mathfrak{g}(L_\mu,L_\lambda\otimes L_\delta
)\simeq\mbox{Hom}_\mathfrak{g}(L_\mu\otimes L_\delta, L_\lambda),
\end{equation}
it should be that $L_\lambda$ is an irreducible submodule or
quotient module of $L_\mu\otimes L_\delta$. Thus $\lambda\in
\mathcal{P}_\mu$, which implies that $\mu\in \mathcal{P}_\lambda$.

Suppose $\mbox{ht }\mu=\mbox{ht }\lambda$.  There can not be a
weight $\nu\succ\mu$ (note that $\mbox{ht } \nu-2\geq \mbox{ht
}\mu=\mbox{ht }\lambda$) such that $[L_\lambda\otimes L_{\delta}:
L_\nu]\neq0$. Thus there should be a weight vector with highest
weight $\mu$. So if we multiply $L_{\delta}$ on the both sides of
(3.13), then on the right side the coefficient of $M_\mu$ is
nonzero. But it is clear by Lemma 4.3 that, for any
$\nu\prec\lambda$ (note that $\mbox{ht } \nu\leq \mbox{ht
}\lambda-2$), the coefficient of $M_\mu$ in $M_\nu\otimes
L_{\delta}$ is zero. So $M_\mu$ appears in $M_\lambda\otimes
L_{\delta}$. It implies that $\mu\in\mathcal {P}_\lambda$ by Lemma
4.3 again.
\end{proof}

\begin{corollary}
For any atypical weight $\lambda\in\mathcal {P}$, if $\mu\in\mathcal
{P}_\lambda$ is also an atypical weight, then $L_\mu$ is a direct
summand in $L_\lambda\otimes L_\delta$ and
\begin{equation}
[L_\lambda\otimes L_\delta:L_\mu]=1.
\end{equation}
\end{corollary}
\begin{proof} For $\mathfrak{g}=\mathfrak{osp}(2m|2)$, just combine Lemmas 4.4, 4.5 and 4.6 except
the case that $\lambda_0=m-1, \lambda_{m-1}=\lambda_m=0$ and
$\mu=\lambda\pm\delta$, whose proof will be seen in Remark 5.4.

For $\mathfrak{g}=\mathfrak{osp}(2m+1|2)$, also combine Lemmas 4.4,
4.5 and 4.6.
\end{proof}

\begin{remark}
Lemma 4.4 and Corollary 4.7 indicate that for any two atypical
weights $\lambda,\mu\in\mathcal {P}$ with
$\lambda^\delta,\mu^\delta\not\in\mathcal {P}$, if there exist
atypical weights
$\lambda^{(0)}=\lambda,\lambda^{(1)},\cdots,\lambda^{(t)}=\mu\in\mathcal
{P}$ such that ${\lambda^{(i)}}^\delta\not\in\mathcal {P}$ and
$\lambda^{(i+1)}\in\mathcal {P}_{\lambda^{(i)}}$, then one can use
Lemma 4.3 iteratively to get $\mbox{ch }L_\mu$ from $\mbox{ch
}L_\lambda$ by a straightforward calculation:
\begin{equation}
L_{\lambda^{(i)}}\xrightarrow[]{\otimes L_\delta}
L_{\lambda^{(i+1)}}.
\end{equation}
Here we require ${\lambda^{(i)}}^\delta\not\in\mathcal {P}$ just
because that we have not done with the case of
$\lambda^\delta\in\mathcal {P}$ in the proof of Lemma 4.5, Lemma 4.6
and Corollary 4.7 yet.
\end{remark}

\section{Character formulae}

The goal of this section is to obtain the character formulae for
$\mathfrak{osp}(n|2)$ ($n=2m$ or $2m+1$) in terms of the characters
of generalized Verma modules.

\subsection{The case when $\lambda_0\leq m-1$}
For any $\lambda\in\mathcal {P}$ with $\lambda_0\leq m-1$, we have
$k\leq \lambda_0\leq m-1$ where $k$ is the maximal number such that
$\lambda_k\neq0$ by Theorem 2.1. Now denote
\begin{equation}
\lambda^{j,q}=\sum_{i=1}^j\lambda_i\epsilon_i+q\epsilon_{j+1}+\sum_{i=j+1}^{\lambda_0}(\lambda_i+1)\epsilon_{i+1}+(j-q)\delta
\end{equation} and
\begin{equation}
\lambda^{j,q}_{-}=\sum_{i=1}^j\lambda_i\epsilon_i+q\epsilon_{j+1}+
\sum_{i=j+1}^{\lambda_0-1}(\lambda_i+1)\epsilon_{i+1}-(\lambda_{\lambda_0}+1)\epsilon_{\lambda_0+1}+(j-q)\delta.
\end{equation}

\begin{lemma} If $\lambda\in\mathcal {P}$ with $\lambda_0\leq m-1$,
then $\lambda\sim\lambda^{j,q}$ for any
$j\in\{0,1,\ldots,\lambda_0\}$ and $q\in\mathbb{Z}$. Moreover, for
$\mathfrak{g}=\mathfrak{osp}(2m|2)$, if $\lambda_0=m-1$, then
$\lambda\sim\lambda^{j,q}\sim\lambda^{j,q}_{-}$ for any
$j\in\{0,1,\ldots,\lambda_0\}$ and $q\in\mathbb{Z}$.
\end{lemma}
\begin{proof}
It is obvious that $(\lambda^{j,q}+\rho,\delta-\epsilon_{j+1})=0$.
Thus we have
$\lambda^{j,q}\sim\lambda^{j,q+1}=\lambda^{j,q}-(\delta-\epsilon_{j+1})$.
Therefore $\lambda^{j,q_1}\sim\lambda^{j,q_2}$ for any
$q_1,q_2\in\mathbb{Z}$. On the other hand, it is easy to check that
$\lambda^{j,\lambda_j+1}=\lambda^{j-1,\lambda_j}$ and
$\lambda=\lambda^{\lambda_0,0}$. Hence $\lambda\sim\lambda^{j,q}$
for any $j\in\{0,1,\ldots,\lambda_0\}$ and $q\in\mathbb{Z}$.

Moreover, when $\lambda_0=m-1$, $(\lambda+\rho,\delta\pm
\epsilon_m)=0$ for $\mathfrak{g}=\mathfrak{osp}(2m|2)$. Notice that
$\lambda_m=0$ because of $\lambda_0\leq m-1$. There is a symmetry
between $\epsilon_m$ and $-\epsilon_m$. So we can also show that
$\lambda\sim\lambda^{j,q}_{-}$.
\end{proof}

\begin{theorem}
If $\mathfrak{g}=\mathfrak{osp}(2m|2)$ and $\lambda\in\mathcal {P}$
with $k\leq \lambda_0\leq m-2$, or
$\mathfrak{g}=\mathfrak{osp}(2m+1|2)$ and $\lambda\in\mathcal {P}$
with $k\leq \lambda_0\leq m-1$, where $k$ is the maximal number such
that $\lambda_k\neq0$, in this case,
$(\lambda+\rho,\delta-\epsilon_{\lambda_0+1})=0$, then
\begin{equation}
L_{\lambda}=M_{\lambda}+
\sum_{q=\lambda_{1}+1}^{\infty}(-1)^{q}M_{\lambda^{0,q}}+
\sum_{j=1}^k\sum_{q=\lambda_{j+1}+1}^{\lambda_j}(-1)^{q}M_{\lambda^{j,q}}.
\end{equation}

If  $\mathfrak{g}=\mathfrak{osp}(2m|2)$ and $\lambda\in\mathcal {P}$
with $k\leq \lambda_0=m-1$ where $k$ is the maximal number such that
$\lambda_k\neq0$, in this case,
$(\lambda+\rho,\delta\pm\epsilon_{m})=0$, then
\begin{equation}
L_{\lambda}= M_{\lambda}+ \sum_{q=\lambda_{1}+1}^{\infty}(-1)^{q}
(M_{\lambda^{0,q}} +M_{\lambda^{0,q}_{-}})+
\sum_{j=1}^k\sum_{q=\lambda_{j+1}+1}^{\lambda_j}(-1)^{q}
(M_{\lambda^{j,q}} +M_{\lambda^{j,q}_{-}}).
\end{equation}
\end{theorem}

\begin{proof} We shall prove only the case of
$\mathfrak{g}=\mathfrak{osp}(2m|2)$ here since the argument for the
case of $\mathfrak{g}=\mathfrak{osp}(2m+1|2)$ is similar.

Use induction on the height of $\lambda$.

When $\mbox{ht}(\lambda)=0$, i.e. $\lambda=0$, the statement holds
by Corollary 3.5.

Suppose the statement holds for $\lambda$.

Multiply $L_\delta$ on the both sides of (5.3) and compute the right
side by Lemma 4.3. Then we can get that

\begin{equation}
L_\lambda\otimes L_\delta=\sum_{\mu\in\mathcal {P}_\lambda,
\mu_0\leq m-2}\left(M_{\mu}+
\sum_{q=\mu_{1}+1}^{\infty}(-1)^{q}M_{\mu^{0,q}}+
\sum_{j=1}^k\sum_{q=\mu_{j+1}+1}^{\mu_j}(-1)^{q}M_{\mu^{j,q}}\right)+
\end{equation}
\begin{equation*}
\sum_{\mu\in\mathcal {P}_\lambda, \mu_0= m-1}\left(M_{\mu}+
\sum_{q=\mu_{1}+1}^{\infty}(-1)^{q} (M_{\mu^{0,q}}
+M_{\mu^{0,q}_{-}})+
\sum_{j=1}^k\sum_{q=\mu_{j+1}+1}^{\mu_j}(-1)^{q} (M_{\mu^{j,q}}
+M_{\mu^{j,q}_{-}})\right).
\end{equation*}

Observe that for any $\mu\in\mathcal {P}$, the coefficient of
$M_\mu$ is $0$ on the right side unless $\mu\in\mathcal
{P}_\lambda$. Moreover, the coefficients of all $M_\mu$ with
$\mu\in\mathcal {P}_\lambda$ are 1. Notice that all weights in
$\mathcal {P}_\lambda$ are in different blocks by Lemma 4.4. Hence
by Remark 3.3, it should be that
\begin{equation} L_\lambda\otimes
L_\delta=\bigoplus_{\mu\in\mathcal {P}_{\lambda}}L_\mu
\quad\mbox{for any $\lambda\in\mathcal {P}$ with $\lambda_0\leq
m-2$.}
\end{equation}
(Notice: Here we can not use Corollary 4.7 to obtain (5.6) directly
because we have not done with this case in the proof of Lemmas 4.5
and 4.6.)

Select all $M_{\mu'}$'s with $\mu'\sim\mu$ for any $\mu\in\mathcal
{P}_{\lambda^+}$ on the right side of (5.5) by Lemma 5.1, then we
get the expression of $L_{\mu}$ for any $\mu\in\mathcal
{P}_{\lambda^+}$ directly.
\end{proof}

\subsection{The case when $\lambda_0\geq m$ and $\lambda^\delta\in\mathcal {P}$}
We can check that for any $\lambda\in\mathcal {P}$ with
$\lambda_0\geq m$, the weight $\lambda^\delta\in\mathcal {P}$ if and
only if $\lambda_0\leq n-2-k$ where $k$ is the maximal number with
$\lambda_k\neq0$.

\begin{theorem}
If $\lambda\in\mathcal {P}$ satisfies that $\lambda_0\geq m$ and
$\lambda^\delta\in\mathcal {P}$, in this case,
$(\lambda+\rho,\delta+\epsilon_{n-1-\lambda_0})=0$, then
\begin{equation}
L_{\lambda}=M_{\lambda}+
\sum_{q=\lambda_{1}+1}^{\infty}(-1)^{q}M_{(\lambda^\delta)^{0,q}}+
\sum_{j=1}^k\sum_{q=\lambda_{j+1}+1}^{\lambda_j}(-1)^{q}M_{(\lambda^\delta)^{j,q}}=M_{\lambda}
-M_{\lambda^\delta}+L_{\lambda^\delta},
\end{equation} where $k$ is the maximal number
with $\lambda_k\neq0$.
\end{theorem}
\begin{proof} 1). If $\mathfrak{g}=\mathfrak{osp}(2m|2)$.

Firstly, we should determine $L_\lambda\otimes L_\delta$ in case of
$\lambda_0=m-1$ and $k\leq m-2$ which induces
$\lambda+\delta,\lambda-\delta\in\mathcal {P}_\lambda$ and
$\lambda+\delta\sim\lambda-\delta$.

Multiply $L_\delta$ on the both sides of (5.4). We have that if
$\lambda_0=m-1$ and $k\leq m-2$, then
\begin{eqnarray*}
L_\lambda\otimes
L_\delta=\left(M_{\lambda+\delta}+M_{\lambda-\delta}+
2\sum_{q=\lambda_{1}+1}^{\infty}(-1)^{q}M_{(\lambda-\delta)^{0,q}}+
2\sum_{j=1}^k\sum_{q=\lambda_{j+1}+1}^{\lambda_j}(-1)^{q}M_{(\lambda-\delta)^{j,q}}\right)\\
+\sum_{\mu\in\mathcal
{P}_\lambda,\mu\neq\lambda\pm\delta}\left(M_{\mu}+
\sum_{q=\mu_{1}+1}^{\infty}(-1)^{q} (M_{\mu^{0,q}}
+M_{\mu^{0,q}_{-}})+
\sum_{j=1}^k\sum_{q=\mu_{j+1}+1}^{\mu_j}(-1)^{q} (M_{\mu^{j,q}}
+M_{\mu^{j,q}_{-}})\right).
\end{eqnarray*}
Since $\lambda+\delta>\lambda-\delta$, we know
\begin{equation}
[L_\lambda\otimes L_\delta: L_{\lambda+\delta}]=1
\end{equation}
by Remark 3.3.

Suppose that
\begin{equation}
[L_{\lambda}\otimes L_\delta: L_{\lambda-\delta}]=x.
\end{equation}(Notice: here we can not use Corollary 4.7 to obtain
$x=1$ directly because we still have not proved that corollary in
this case.)

Thus
\begin{eqnarray*}
&&L_{\lambda+\delta}+xL_{\lambda-\delta}\\&&=M_{\lambda+\delta}+M_{\lambda-\delta}+
2\sum_{q=\lambda_{1}+1}^{\infty}(-1)^{q}M_{(\lambda-\delta)^{0,q}}+
2\sum_{j=1}^k\sum_{q=\lambda_{j+1}+1}^{\lambda_j}(-1)^{q}M_{(\lambda-\delta)^{j,q}}\\
&&=M_{\lambda+\delta}+M_{\lambda-\delta}+2L_{\lambda-\delta}-2M_{\lambda-\delta}
\end{eqnarray*}
That is
\begin{equation}
L_{\lambda+\delta}=M_{\lambda+\delta}-M_{\lambda-\delta}-(x-2)L_{\lambda-\delta}.
\end{equation}
Since $[L_{\lambda-\delta}\otimes L_\delta: L_\lambda]=1$ by (5.6)
and $[(M_{\lambda+\delta}-M_{\lambda-\delta})\otimes L_\delta:
L_\lambda]=0$ by Lemma 4.3, we have
\begin{equation}
0\leq [L_{\lambda+\delta}\otimes L_\delta:
L_\lambda]=[(M_{\lambda+\delta}-M_{\lambda-\delta}-(x-2)L_{\lambda-\delta})\otimes
L_\delta: L_\lambda]=2-x.
\end{equation}
That is,  $x\leq 2$.

Again by (5.6), we have
\begin{equation}
\dim\mbox{Hom}_\mathfrak{g}(L_{\lambda-\delta}, L_\lambda\otimes
L_\delta)=\dim\mbox{Hom}_\mathfrak{g}(L_{\lambda-\delta}\otimes
L_\delta, L_\lambda)=1.
\end{equation}
Thus $x\neq0$.

If $x=2$, then \begin{equation}
L_{\lambda+\delta}=M_{\lambda+\delta}-M_{\lambda-\delta}.
\end{equation}
But it is impossible because of Corollary 2.5.

So it must be that $x=1$. Thus
\begin{equation}
L_{\lambda+\delta}=M_{\lambda+\delta}-M_{\lambda-\delta}+L_{\lambda-\delta}.
\end{equation} That is, for the case of $\lambda_0=m$, formula
(5.7) holds.

Below we shall prove the statement by induction on $\lambda_0$.
Suppose that (5.7) holds for the case of $m\leq\lambda_0<t$. Now
take $\lambda$ with $\lambda_0=t-1$. By induction assumption, Lemma
4.3 and (5.6), we have
\begin{eqnarray*}
L_\lambda\otimes
L_\delta&=&(M_\lambda-M_{\lambda^\delta}+L_{\lambda^\delta})\otimes
L_\delta\\
&=&M_{\lambda+\delta}-M_{\lambda^\delta-\delta}+L_{\lambda^\delta-\delta}+\cdots
\end{eqnarray*}
where ``$\cdots$'' is a sum of $M_\mu$ or $L_\mu$ with
$\mu\not\sim\lambda$. Hence by Remark 3.3, it should be that
\begin{equation}
[L_\lambda\otimes L_\delta: L_{\lambda+\delta}]=1.
\end{equation}
Moreover, if we assume that
\begin{equation}
[L_\lambda\otimes L_\delta: L_{\lambda^\delta-\delta}]=x,
\end{equation} then
\begin{equation}
L_{\lambda+\delta}=M_{\lambda+\delta}-M_{\lambda^\delta-\delta}+(1-x)L_{\lambda^\delta-\delta}.
\end{equation}
In order to determine $x$, we suppose that
\begin{equation}
[L_{\lambda+\delta}\otimes L_\delta: L_\lambda]=y_1,\quad
[L_{\lambda+\delta}\otimes L_\delta: L_{\lambda^\delta}]=y_2
\end{equation} and multiply $L_\delta$ on the both sides
of (5.17). Taking the terms $L_\mu$ and $M_\mu$ with
$\mu\sim\lambda$, we obtain
\begin{equation}
y_1L_\lambda+y_2L_{\lambda^\delta}=M_\lambda-M_{\lambda^\delta}+(1-x)L_{\lambda^\delta}=L_\lambda-xL_{\lambda^\delta}.
\end{equation}
By Remark 3.3, we have $y_1=1$. Notice that $x,y_2\geq0$, so it must
be that $x=y_2=0$. Thus
\begin{equation}
L_{\lambda+\delta}=M_{\lambda+\delta}-M_{\lambda^\delta-\delta}+L_{\lambda^\delta-\delta}.
\end{equation} That is, for any $\lambda\in\mathcal {P}$ with
$\lambda_0=t$, we also have
\begin{equation}
L_\lambda=M_\lambda-M_{\lambda^\delta}+L_{\lambda^\delta}.
\end{equation}

2). If $\mathfrak{g}=\mathfrak{osp}(2m+1|2)$.

The difference here from the argument for the case of
$\mathfrak{g}=\mathfrak{osp}(2m|2)$ is to determine $L_\lambda$ when
$\lambda_0=m$.

For any $\lambda\in\mathcal {P}$ with $\lambda^\delta\in\mathcal
{P}$ and $\lambda_0=m$. It is obvious that
$\lambda^\delta=\lambda-\delta$ and $\lambda_m=0$. By Theorem 5.2,
we have
\begin{equation}
L_{\lambda-\delta}=M_{\lambda-\delta}+\sum_{q=\lambda_{1}+1}^{\infty}(-1)^{q}M_{(\lambda-\delta)^{0,q}}+
\sum_{j=1}^k\sum_{q=\lambda_{j+1}+1}^{\lambda_j}(-1)^{q}M_{(\lambda-\delta)^{j,q}}.
\end{equation}

On the right side of (5.22), the coefficient of $\epsilon_m$ in any
$M_{\mu}$ except $M_{\lambda-\delta}$ is $0$. Hence if we multiply
$L_\delta$ on both sides of (5.22) by (4.10) and select the terms
$L_\mu$ and $M_\mu$ with $\mu\sim\lambda$, there comes
\begin{equation}
x_1L_\lambda+x_2L_{\lambda-\delta}=M_\lambda+\sum_{q=\lambda_{1}+1}^{\infty}(-1)^{q}M_{(\lambda-\delta)^{0,q}}+
\sum_{j=1}^k\sum_{q=\lambda_{j+1}+1}^{\lambda_j}(-1)^{q}M_{(\lambda-\delta)^{j,q}},
\end{equation}
where
\begin{equation}
x_1=[L_{\lambda-\delta}\otimes L_\delta: L_{\lambda}]\geq0;\quad
x_2=[L_{\lambda-\delta}\otimes L_\delta: L_{\lambda-\delta}]\geq0.
\end{equation}

Express $L_{\lambda-\delta}$ with the form (5.3) and multiply
$L_\delta$ on the both side, it is clear that $x_1=1$ by Remark 3.3.
In order to determine $x_2$, we multiply $L_\delta$ on both sides of
(5.23) and get that the coefficient of $M_\lambda$ on the right side
is $0$. Hence also by Remark 3.3, we have
\begin{equation}
0=[(L_\lambda+x_2L_{\lambda-\delta})\otimes L_\delta: L_\lambda]\geq
x_2.
\end{equation}
Thus $x_2=0$. That is, for any $\lambda\in\mathcal {P}$ with
$\lambda^\delta\in\mathcal {P}$ and $\lambda_0=m$,
\begin{equation}
L_\lambda=M_\lambda-M_{\lambda^\delta}+L_{\lambda^\delta}.
\end{equation}

The rest argument is as the same as the case of
$\mathfrak{g}=\mathfrak{osp}(2m|2)$.
\end{proof}

\begin{remark}
We can not use Lemma 4.5, Lemma 4.6 and Corollary 4.7 in the proof
of Theorems 5.2 and 5.3 because we did not deal with the case of
$\lambda^\delta\in\mathcal {P}$ there. Now according to the
calculation in proof of the above two theorems, we can see that
Lemmas 4.5 and 4.6 do hold in case of $\lambda^\delta\in\mathcal
{P}$. To complete the proof of Corollary 4.7, we need to deal with
the case that $\mathfrak{g}=\mathfrak{osp}(2m|2)$, $\lambda_0=m-1,
\lambda_{m-1}=\lambda_m=0$ and $\mu=\lambda\pm\delta$. We should
prove that the ``$+$'' in $L_{\lambda}\otimes
L_\delta=(L_{\lambda+\delta}+L_{\lambda-\delta})\oplus\cdots$ can be
changed to ``$\oplus$''. In fact, we have known that by (5.6)
\begin{equation}
\dim\mbox{Hom}_{\mathfrak{g}}(L_{\lambda-\delta}\otimes L_\delta,
L_\lambda)=\dim\mbox{Hom}_{\mathfrak{g}}(L_\lambda,
L_{\lambda-\delta}\otimes L_\delta)=1.
\end{equation}
Thus
\begin{equation}
\dim\mbox{Hom}_{\mathfrak{g}}(L_{\lambda-\delta}, L_\lambda\otimes
L_\delta)=\dim\mbox{Hom}_{\mathfrak{g}}(L_\lambda\otimes L_\delta,
L_{\lambda-\delta})=1.
\end{equation}
Therefore the $L_{\lambda-\delta}$ is exactly a direct summand.

From now on, we can use Lemma 4.5, Lemma 4.6 and Corollary 4.7
freely.
\end{remark}

\subsection{When $\mathfrak{g}=\mathfrak{osp}(2m|2)$ and $\lambda^\delta\not\in\mathcal {P}$}
Now the atypical weight $\lambda$ is
$(\delta+\epsilon_i)$-atypical for certain $1\leq i\leq m$ (it may
be that $\lambda$ is $(\delta-\epsilon_m)$-atypical, which is as the
same as $(\delta+\epsilon_m)$-atypical essentially).

\subsubsection{\bf The case that $\lambda$ is
$(\delta+\epsilon_m)$-atypical.}

Now $\lambda=\lambda_0\delta+\sum_{i=1}^{m}\lambda_i\epsilon_i$
satisfies that $\lambda_0=m-1+\lambda_m$ and $\lambda_m>0$.

It is clear that $\mu=\lambda-\epsilon_m$ is a typical dominant
integral weight. So by Corollary 2.5, we have
\begin{equation}
L_\mu=M_\mu-M_{\mu^\delta}.
\end{equation}
Consider the tensor module $L_\mu\otimes L_\delta$.

\noindent{\bf 1). For $\lambda_m=1$.}

There are three weights  $\mu+\epsilon_m, \mu-\epsilon_m,
\mu-\delta\in\mathcal {P}_\mu$ satisfying
$\mu+\epsilon_m\sim\mu-\epsilon_m\sim\mu-\delta$. Thanks to Lemma
4.5,
\begin{equation}
[L_\mu\otimes L_\delta: L_{\mu+\epsilon_m}]=[L_\mu\otimes L_\delta:
L_{\mu-\epsilon_m}]=1.
\end{equation}
Moreover, we have
\begin{equation}
(M_\mu-M_{\mu^\delta})\otimes
L_\delta=M_{\mu+\epsilon_m}-M_{\mu^\delta+\epsilon_m}+M_{\mu-\epsilon_m}-M_{\mu^\delta-\epsilon_m}+\cdots
\end{equation} where ``$\cdots$'' is a sum of $M_\nu$ with
$\nu\not\sim\mu+\epsilon_m$.

Assume $[L_\mu\otimes L_\delta: L_{\mu-\delta}]=x$, where $x>0$
because of
\begin{equation}
\dim\mbox{Hom}_{\mathfrak{g}}(L_\mu\otimes
L_\delta,L_{\mu-\delta})=\dim\mbox{Hom}_{\mathfrak{g}}(L_\mu,L_{\mu-\delta}\otimes
L_\delta)=1
\end{equation} by Lemma 4.5.
Then
\begin{equation}
L_{\mu+\epsilon_m}+L_{\mu-\epsilon_m}+xL_{\mu-\delta}=
M_{\mu+\epsilon_m}-M_{\mu^\delta+\epsilon_m}+M_{\mu-\epsilon_m}-M_{\mu^\delta-\epsilon_m}.
\end{equation}

We can compute by (4.9) that
\begin{equation}
[(M_{\mu+\epsilon_m}-M_{\mu^\delta+\epsilon_m}+M_{\mu-\epsilon_m}-M_{\mu^\delta-\epsilon_m})\otimes
L_\delta: L_\mu]=2.
\end{equation}
On the other hand,
\begin{equation}
[(L_{\mu+\epsilon_m}+L_{\mu-\epsilon_m}+xL_{\mu-\delta})\otimes
L_\delta : L_\mu]\geq[xL_{\mu-\delta}\otimes L_\delta : L_\mu]=x.
\end{equation}
Hence $x=1$ or $2$ by (5.33), (5.34) and (5.35).

But it is impossible that $x=1$. If so, then by (5.33) and by the
symmetry between $L_{\mu+\epsilon_m}$ and $L_{\mu-\epsilon_m}$, it
should be that
\begin{equation}
L_{\mu+\epsilon_m}=\frac{1}{2}L_{\mu-\delta}+\cdots=\frac{1}{2}M_{\mu-\delta}+\cdots,
\end{equation} which is a contradiction to Remark 3.2.

Now we know that $x=2$. Therefore
\begin{equation}
L_{\mu+\epsilon_m}=M_{\mu+\epsilon_m}-M_{\mu^\delta+\epsilon_m}-L_{\mu-\delta}
\end{equation}
and
\begin{equation}
L_{\mu-\epsilon_m}=M_{\mu-\epsilon_m}-M_{\mu^\delta-\epsilon_m}-L_{\mu-\delta}.
\end{equation}
Here we should explain why it is not that
$L_{\mu+\epsilon_m}=M_{\mu+\epsilon_m}-M_{\mu^\delta-\epsilon_m}-L_{\mu-\delta}$.
Assume that $p$ is the smallest positive number such that
$\mu+\epsilon_m+p\delta$ is typical. Taking
$\lambda^{(i)}=\mu+\epsilon_m+i\delta$ in Remark 4.8, it is easy for
us to get the character of $L_{\mu+\epsilon_m+p\delta}$ from
$L_{\mu+\epsilon_m}$. Only the expression (5.37) can induce the
correct formula
$L_{\mu+\epsilon_m+p\delta}=M_{\mu+\epsilon_m+p\delta}-M_{(\mu+\epsilon_m+p\delta)^\delta}$.

Equivalent to (5.37), we obtain
\begin{equation}
L_\lambda=M_\lambda-M_{\lambda^\delta}-L_{\lambda-\delta-\epsilon_m}.
\end{equation}

By the way, we can compute directly by (5.39) that
\begin{equation}
[L_{\lambda}\otimes L_\delta:
L_{\lambda-\delta}]=[L_{\lambda}\otimes L_\delta:
L_{\lambda-\epsilon_m}]=0.
\end{equation}

\noindent{\bf 2). For any $\lambda_m$.}

We shall prove that (5.39) and (5.40) still hold for any
$\lambda_m\geq2$ by induction on $\lambda_m$.

Suppose
\begin{equation}
[L_\mu\otimes L_\delta:
L_{\lambda-i\delta-i\epsilon_m}]=x_i,\quad(1\leq i\leq\lambda_m).
\end{equation}
Then if we multiply $L_\delta$ on both sides of (5.29) and choose
the terms $L_\nu$ and $M_\nu$ with $\nu\sim\lambda$, then
\begin{equation}
L_\lambda+\sum_{i=1}^{\lambda_m}x_iL_{\lambda-i\delta-i\epsilon_m}=
M_\lambda-M_{\lambda^\delta}+M_{\lambda-\delta-\epsilon_m}-
M_{(\lambda-\delta-\epsilon_m)^\delta}.
\end{equation}

By induction assumption, $L_{\lambda-\delta-\epsilon_m}$ can
expressed as form (5.39). Thus (5.42) becomes to
\begin{equation}
L_\lambda=M_\lambda-M_{\lambda^\delta}+(1-x_1)L_{\lambda-\delta-\epsilon_m}
+(1-x_2)L_{\lambda-2\delta-2\epsilon_m}
-\sum_{i=3}^{\lambda_m}x_iL_{\lambda-i\delta-i\epsilon_m}.
\end{equation}
Notice that $\lambda-(i+1)\delta-i\epsilon_m$ ($1\leq i\leq
\lambda_m-1$) are all typical weights. Thus if we multiply
$L_\delta$ on both sides of (5.43) and choose the terms in the block
corresponding to $\lambda-(i+1)\delta-i\epsilon_m$ ($1\leq i\leq
\lambda_m-1$), there comes $x_2=1$, $x_3=\cdots=x_{\lambda_m}=0$ by
Lemma 4.5, Lemma 4.6 and induction assumption. Thus
\begin{equation}
L_\lambda=M_\lambda-M_{\lambda^\delta}+(1-x_1)L_{\lambda-\delta-\epsilon_m}.
\end{equation}
As the same as (5.32)-(5.35), we know $x_1=1$ or $2$. Thanks to
Corollary 2.5, it should be that $x_1=2$. Thus we have showed that
(5.39) and (5.40) hold for any $\lambda_m$.

\subsubsection{\bf The case that $\lambda$ is
$(\delta-\epsilon_m)$-atypical.}

By the symmetry between $(\delta-\epsilon_m)$ and
$(\delta+\epsilon_m)$, one can get the following formulae at once.
\begin{equation}
L_{\lambda}= M_{\lambda}-
M_{\lambda^\delta}-L_{\lambda-\delta+\epsilon_m}
\end{equation} and
\begin{equation}
[L_{\lambda}\otimes L_\delta:
L_{\lambda-\delta}]= [L_{\lambda}\otimes L_\delta:
L_{\lambda+\epsilon_m}]=0.
\end{equation}

\subsubsection{\bf The case that $\lambda$ is
$(\delta+\epsilon_k)$-atypical with
$\lambda_{k+1}=\lambda_{k+2}=\cdots=\lambda_{m}=0$ ($1\leq k\leq
m-1$)}

Now we have $\lambda_0=2m-k+\lambda_k-1$, $\lambda_{k}>0$ and
$\lambda_{k+1}=\lambda_{k+2}=\cdots=\lambda_{m}=0$.

It is obvious that $\mu=\lambda-\epsilon_k$ is a typical dominant
integral weight, and there are two weights $\mu-\delta$ and
$\lambda=\mu+\epsilon_k$ in $\mathcal {P}_\mu$ such that
$\mu-\delta\sim\lambda$.

Suppose
\begin{equation}
[L_\mu\otimes L_\delta: L_{\lambda-i\delta-i\epsilon_k}]=x_i
\quad(1\leq i \leq \lambda_k)
\end{equation}
and
\begin{equation}
[L_\mu\otimes L_\delta:
L_{(\lambda-\lambda_k\delta-\lambda_k\epsilon_k)^\delta}]=x_0.
\end{equation}
Then similar to (5.42), we have
\begin{equation}
L_\lambda+\sum_{i=1}^{\lambda_k}x_iL_{\lambda-i\delta-i\epsilon_k}+x_0L_{(\lambda-\lambda_k\delta-\lambda_k\epsilon_k)^\delta}
=M_\lambda-M_{\lambda^\delta}+M_{\lambda-\delta-\epsilon_k}-M_{(\lambda-\delta-\epsilon_k)^\delta}.
\end{equation}
Imitating (5.32)-(5.35), we can also show that $x_1=1$ or $2$.

\noindent{\bf 1). For $\lambda_k=1$.}

Now (5.49) can be simplified to
\begin{equation}
L_\lambda+x_1L_{\lambda-\delta-\epsilon_k}+x_0L_{(\lambda-\delta-\epsilon_k)^\delta}
=M_\lambda-M_{\lambda^\delta}+M_{\lambda-\delta-\epsilon_k}-M_{(\lambda-\delta-\epsilon_k)^\delta}.
\end{equation}
Notice that by Theorem 5.3,
\begin{equation}
L_{\lambda-\delta-\epsilon_k}=M_{\lambda-\delta-\epsilon_k}-M_{(\lambda-\delta-\epsilon_k)^\delta}+L_{(\lambda-\delta-\epsilon_k)^\delta}.
\end{equation}
It immediately follows from (5.50) and (5.51) that
\begin{equation}
L_\lambda=M_\lambda-M_{\lambda^\delta}-(x_1-1)L_{\lambda-\delta-\epsilon_k}
-(x_0+1)L_{(\lambda-\delta-\epsilon_k)^\delta}.
\end{equation}

If $x_1=[L_{\lambda-\epsilon_k}\otimes L_\delta:
L_{\lambda-\delta-\epsilon_k}]=1$. By Lemma 4.5, we have
\begin{equation}
\dim\mbox{Hom}_{\mathfrak{g}}(L_{\lambda-\delta-\epsilon_k}\otimes
L_\delta, L_{\lambda-\epsilon_k})=\dim\mbox{Hom}_{\mathfrak{g}}(
L_{\lambda-\epsilon_k},L_{\lambda-\delta-\epsilon_k}\otimes
L_\delta)=1.
\end{equation}
That is,
\begin{equation}
\dim\mbox{Hom}_{\mathfrak{g}}(L_{\lambda-\delta-\epsilon_k},
L_{\lambda-\epsilon_k}\otimes
L_\delta)=\dim\mbox{Hom}_{\mathfrak{g}}(
L_{\lambda-\epsilon_k}\otimes
L_\delta,L_{\lambda-\delta-\epsilon_k})=1.
\end{equation} Thus $L_{\lambda-\delta-\epsilon_k}$ should be a direct
summand in $L_{\lambda-\epsilon_k}\otimes L_\delta$. But
$L_{(\lambda-\delta-\epsilon_k)^\delta}$ can not be a submodule or
quotient module of $L_{\lambda-\epsilon_k}\otimes L_\delta$ by Lemma
4.6. It has to be that $x_0=0$. So
\begin{equation}
L_\lambda=M_\lambda-M_{\lambda^\delta}-L_{(\lambda-\delta-\epsilon_k)^\delta}.
\end{equation}
However, this is not the case. Take
\begin{equation}
\lambda^{(0)}=\lambda,\quad
\lambda^{(i)}=\lambda+\sum_{j=1}^{i}\epsilon_{k+j}\quad (1\leq i\leq
m-k-1).
\end{equation}
Then using the method introduced in Remark 4.8, the expression
(5.55) would lead us arriving at
\begin{equation}
L_{\lambda+\sum_{i=k+1}^{m-1}\epsilon_i}=
M_{\lambda+\sum_{i=k+1}^{m-1}\epsilon_i}-M_{\lambda^\delta+\sum_{i=k+1}^{m-1}\epsilon_i}
-L_{(\lambda-\delta-\epsilon_k)^\delta+(m-k-1)\delta}.
\end{equation}
Multiply $L_\delta$ on the both sides of the above equation and
fetch the terms $L_\nu$ and $M_\nu$ with
$\nu\sim\lambda+\sum_{i=k+1}^{m-1}\epsilon_i+\epsilon_m\sim\lambda+\sum_{i=k+1}^{m-1}\epsilon_i-\epsilon_m$.
We get
\begin{eqnarray}
L_{\lambda+\sum_{i=k+1}^{m-1}\epsilon_i+\epsilon_m}+L_{\lambda+\sum_{i=k+1}^{m-1}\epsilon_i-\epsilon_m}
=-L_{(\lambda-\delta-\epsilon_k)^\delta+(m-k)\delta}+\quad\quad\quad\quad\\
M_{\lambda+\sum_{i=k+1}^{m-1}\epsilon_i+\epsilon_m}-M_{\lambda^\delta+\sum_{i=k+1}^{m-1}\epsilon_i+\epsilon_m}
+M_{\lambda+\sum_{i=k+1}^{m-1}\epsilon_i-\epsilon_m}-M_{\lambda^\delta+\sum_{i=k+1}^{m-1}\epsilon_i-\epsilon_m}\nonumber
\end{eqnarray}
and then
\begin{equation}
L_{\lambda+\sum_{i=k+1}^{m-1}\epsilon_i+\epsilon_m}=-\frac{1}{2}L_{(\lambda-\delta-\epsilon_k)^\delta+(m-k)\delta}+\cdots
=-\frac{1}{2}M_{(\lambda-\delta-\epsilon_k)^\delta+(m-k)\delta}+\cdots,
\end{equation}
which is a contradiction to Remark 3.2.

Thus we get $x_1=2$ and
\begin{equation}
L_\lambda=M_\lambda-M_{\lambda^\delta}-L_{\lambda-\delta-\epsilon_k}-(x_0+1)L_{(\lambda-\delta-\epsilon_k)^\delta}.
\end{equation}
Now the expression (5.60) would lead us arriving at
\begin{eqnarray}
L_{\lambda+\sum_{i=k+1}^{m-1}\epsilon_i+\epsilon_m}+L_{\lambda+\sum_{i=k+1}^{m-1}\epsilon_i-\epsilon_m}
=-(x_0+2)L_{(\lambda-\delta-\epsilon_k)^\delta+(m-k)\delta}+\\
M_{\lambda+\sum_{i=k+1}^{m-1}\epsilon_i+\epsilon_m}-M_{\lambda^\delta+\sum_{i=k+1}^{m-1}\epsilon_i+\epsilon_m}
+M_{\lambda+\sum_{i=k+1}^{m-1}\epsilon_i-\epsilon_m}-M_{\lambda^\delta+\sum_{i=k+1}^{m-1}\epsilon_i-\epsilon_m}\nonumber.
\end{eqnarray}

Fortunately, there is another way to get
$L_{\lambda+\sum_{i=k+1}^{m-1}\epsilon_i+\epsilon_m}+L_{\lambda+\sum_{i=k+1}^{m-1}\epsilon_i-\epsilon_m}$:
Take
\begin{equation}
\lambda^{(i)}=\lambda+\sum_{j=k+1}^{m}\epsilon_j+(k-m+i)\delta,\quad(0\leq
i\leq m-k).
\end{equation}
By (5.39), we have
\begin{equation}
L_{\lambda^{(0)}}=M_{\lambda^{(0)}}-
M_{{\lambda^{(0)}}^\delta}-L_{\lambda^{(0)}-\delta-\epsilon_m}.
\end{equation}
Applying Remark 4.8, we obtain
\begin{equation}
L_{\lambda+\sum_{i=k+1}^{m}\epsilon_i}=L_{\lambda^{(m-k)}}=
M_{\lambda+\sum_{i=k+1}^{m-1}\epsilon_i+\epsilon_m}-
M_{\lambda^\delta+\sum_{i=k+1}^{m-1}\epsilon_i+\epsilon_m}-L_{(\lambda-\delta-\epsilon_k)^\delta+(m-k)\delta}
\end{equation}
Similarly,
\begin{equation}
L_{\lambda+\sum_{i=k+1}^{m-1}\epsilon_i-\epsilon_m}=M_{\lambda+\sum_{i=k+1}^{m-1}\epsilon_i-\epsilon_m}-
M_{\lambda^\delta+\sum_{i=k+1}^{m-1}\epsilon_i-\epsilon_m}-L_{(\lambda-\delta-\epsilon_k)^\delta+(m-k)\delta}.
\end{equation}

Comparing (5.61) with (5.64) and (5.65) takes us to $x_0=0$. Thus
\begin{equation}
L_\lambda=M_\lambda-M_{\lambda^\delta}-L_{\lambda-\delta-\epsilon_k}-L_{(\lambda-\delta-\epsilon_k)^\delta}.
\end{equation}

By the way, we can compute directly by (5.66) that
\begin{equation}
[L_{\lambda}\otimes L_\delta:
L_{\lambda-\epsilon_k}]=[L_{\lambda}\otimes L_\delta:
L_{\lambda-\delta}]=0
\end{equation}

\noindent{\bf 2). For $\lambda_k=2$.}

Now (5.49) can be simplified to
\begin{equation}
L_\lambda+x_1L_{\lambda-\delta-\epsilon_k}+x_2L_{\lambda-2\delta-2\epsilon_k}+x_0L_{(\lambda-2\delta-2\epsilon_k)^\delta}
=M_\lambda-M_{\lambda^\delta}+M_{\lambda-\delta-\epsilon_k}-M_{(\lambda-\delta-\epsilon_k)^\delta},
\end{equation}
where $L_{\lambda-\delta-\epsilon_k}$ can be expressed by (5.66).
Hence
\begin{equation}
L_\lambda=M_\lambda-M_{\lambda^\delta}+(1-x_1)L_{\lambda-\delta-\epsilon_k}
+(1-x_2)L_{\lambda-2\delta-2\epsilon_k}+(1-x_0)L_{(\lambda-2\delta-2\epsilon_k)^\delta}.
\end{equation}
Observe that $\lambda-\delta-2\epsilon_k$ is a typical weight. Thus
if we multiply $L_\delta$ on both sides of (5.69) and choose the
terms in the block corresponding to $\lambda-\delta-2\epsilon_k$,
then we can get $x_2=1$. But $L_{\lambda-2\delta-2\epsilon_k}$ can
not be a submodule or quotient module of
$L_{\lambda-\epsilon_k}\otimes L_\delta$ by Lemma 4.6. It means that
$x_1=2$ (the argument is similar to (5.53),(5.54)). Now (5.69)
becomes to
\begin{equation}
L_\lambda=M_\lambda-M_{\lambda^\delta}-L_{\lambda-\delta-\epsilon_k}
+(1-x_0)L_{(\lambda-2\delta-2\epsilon_k)^\delta}.
\end{equation}
Hence applying Remark 4.8 and using the symmetry between
$\pm\epsilon_m$, we can easily get from (5.70) that
\begin{eqnarray}&&
L_{\lambda+\sum_{i=k+1}^m2\epsilon_i}=M_{\lambda+\sum_{i=k+1}^m2\epsilon_i}
-M_{\lambda^\delta+\sum_{i=k+1}^m2\epsilon_i}\\\nonumber&&
-L_{\lambda-(m-k+1)\delta-\epsilon_k+\sum_{i=k+1}^m\epsilon_i}
+\frac{1-x_0}{2}L_{(m-1)\delta+\sum_{i=1}^{k-1}\lambda_i\epsilon_i+\sum_{i=k}^{m-1}\epsilon_i}
\end{eqnarray}

On the other hand, we can also get
\begin{equation}
L_{\lambda+\sum_{i=k+1}^m2\epsilon_i}=M_{\lambda+\sum_{i=k+1}^m2\epsilon_i}
-M_{\lambda^\delta+\sum_{i=k+1}^m2\epsilon_i}-L_{\lambda-(m-k+1)\delta-\epsilon_k+\sum_{i=k+1}^m\epsilon_i}
\end{equation}
from
\begin{eqnarray}
&&\quad\quad L_{\lambda-(m-k)\delta+\sum_{i=k+1}^m2\epsilon_i}=\\
\nonumber&&M_{\lambda-(m-k)\delta+\sum_{i=k+1}^m2\epsilon_i}
-M_{\lambda^\delta+(m-k)\delta+\sum_{i=k+1}^m2\epsilon_i}
-L_{\lambda-(m-k+1)\delta+\sum_{i=k+1}^{m-1}2\epsilon_i+\epsilon_m},
\end{eqnarray}
which we have got in subsection 5.3.1 before.

Comparing (5.71) with (5.72), we obtain $x_0=1$. Thus (5.70) can be
rewritten as
\begin{equation}
L_{\lambda}= M_{\lambda}-M_{\lambda^\delta}
-L_{\lambda-\delta-\epsilon_k}.
\end{equation}
Also we can compute directly by (5.74) that
\begin{equation}
[L_{\lambda}\otimes L_\delta:
L_{\lambda-\epsilon_k}]=[L_{\lambda}\otimes L_\delta:
L_{\lambda-\delta}]=0.
\end{equation}

\noindent {\bf 3). For any $\lambda_k$.}

We shall use induction on $\lambda_k$ to prove that (5.74) and
(5.75) still hold for any $\lambda_k\geq2$.

By induction assumption, (5.49) can become to
\begin{equation}
L_\lambda=M_\lambda-M_{\lambda^\delta}+\sum_{i=1}^2
(1-x_i)L_{\lambda-i\delta-i\epsilon_k}
-\sum_{i=3}^{\lambda_k}x_iL_{\lambda-i\delta-i\epsilon_k}
-x_0L_{(\lambda-\lambda_k\delta-\lambda_k\epsilon_k)^\delta}.
\end{equation}
As the same as before, If we multiply $L_\delta$ on both sides of
(5.76) and fetch the terms in the block corresponding to
$\lambda-i\delta-(i+1)\epsilon_k$, ($1\leq i\leq\lambda_k-1$). Then
we can obtain that $x_2=1$, $x_3=\cdots=x_{\lambda_k}=0$. Since
$x_2=1\neq0$, it must be that $x_1=2$ (the argument is also similar
to (5.53), (5.54)). Thus
\begin{equation}
L_\lambda=M_\lambda-M_{\lambda^\delta}-L_{\lambda-\delta-\epsilon_k}
-x_0L_{(\lambda-\lambda_k\delta-\lambda_k\epsilon_k)^\delta}.
\end{equation}
Now imitating (5.71)-(5.73), we can show that $x_0=0$.

Therefore (5.74) and (5.75) still hold for any $\lambda_k\geq2$.

\subsubsection{\bf The case that $\lambda$ is
$(\delta+\epsilon_k)$-atypical} Thanks to Remark 4.8, now there is
no difficulty for us to get
$L_{(2m-k+\lambda_k-1)\delta+\sum_{i=1}^{m}\lambda_i\epsilon_i}$
from
$L_{(2m-k+\lambda_k-1)\delta+\sum_{i=1}^{k}\lambda_i\epsilon_i}$. We
shall state the theorem in section 5.5 together with the case
$\mathfrak{g}=\mathfrak{osp}(2m+1|2)$.

\subsection{When $\mathfrak{g}=\mathfrak{osp}(2m+1|2)$ and $\lambda^\delta\not\in\mathcal {P}$}
The atypical weight $\lambda$ is $(\delta+\epsilon_i)$-atypical for
certain $1\leq i\leq m$.

\subsubsection{\bf The case that $\lambda$ is
$(\delta+\epsilon_k)$-atypical with
$\lambda_{k+1}=\lambda_{k+2}=\cdots=\lambda_{m}=0$ ($1\leq k\leq
m$)}

It must be that $\lambda_0=m+\lambda_k-k$, $\lambda_k>0$ and
$\lambda_{k+1}=\lambda_{k+2}=\cdots=\lambda_{m}=0$. It can be
checked easily that $\mu=\lambda-\epsilon_k\in\mathcal {P}$ is
typical. Hence
\begin{equation}
L_\mu=M_\mu-M_{\mu^\delta}.
\end{equation}

By Lemma 4.5,
\begin{equation}
[L_\mu\otimes L_\delta: L_\lambda]=1.
\end{equation}
Assume
\begin{equation}
[L_\mu\otimes L_\delta: L_{\lambda-i(\delta+\epsilon_k)}]=x_i,
\quad(i=1,2,\ldots,\lambda_k)
\end{equation}
and
\begin{equation}
[L_\mu\otimes L_\delta:
L_{(\lambda-\lambda_k\delta-\lambda_k\epsilon_k)^\delta}]=x_0.
\end{equation}
Thus if we multiply $L_\delta$ on both sides of (5.78) and fetch the
terms $L_\nu$ and $M_\nu$ with $\nu\sim\lambda$, then we get
\begin{equation}
L_\lambda+\sum_{i=1}^{\lambda_k}x_iL_{\lambda-i(\delta+\epsilon_k)}
+x_0L_{(\lambda-\lambda_k\delta-\lambda_k\epsilon_k)^\delta}=
M_\lambda+M_{\lambda-\delta-\epsilon_k}-M_{\lambda^\delta}-M_{(\lambda-\delta-\epsilon_k)^\delta}.
\end{equation}

Since $\mu$ is typical, we have that
\begin{equation}
[L_{\lambda-\delta-\epsilon_k}\otimes L_\delta: L_\mu]=1
\end{equation} by Lemma 4.5, and have that
\begin{equation}
[L_{\lambda-i(\delta+\epsilon_k)}\otimes L_\delta:
L_\mu]=[L_{(\lambda-\delta-\epsilon_k)^\delta}\otimes L_\delta:
L_\mu]=0,\quad(i=2,3,\ldots,\lambda_k)
\end{equation} by Lemma 4.6.

It is easy to calculate that
\begin{equation}
[(M_\lambda+M_{\lambda-\delta-\epsilon_k}-M_{\lambda^\delta}-M_{(\lambda-\delta-\epsilon_k)^\delta})\otimes
L_\delta : L_\mu]=2.
\end{equation}

So combining (5.82)-(5.85), there comes
\begin{equation}
x_1\leq[L_\lambda\otimes L_\delta: L_\mu]+x_1=2.
\end{equation}

Moreover, since
\begin{equation}
\dim\mbox{Hom}_{\mathfrak{g}}(L_\mu\otimes
L_\delta,L_{\mu-\delta})=\dim\mbox{Hom}_{\mathfrak{g}}(L_\mu,L_{\mu-\delta}\otimes
L_\delta)=1,
\end{equation} it should be that $x_1>0$.

\noindent{\bf 1). For $\lambda_k=1$.}

Now (5.82) becomes to
\begin{equation}
L_\lambda+x_1L_{\lambda-\delta-\epsilon_k}
+x_0L_{(\lambda-\delta-\epsilon_k)^\delta}=
M_\lambda+M_{\lambda-\delta-\epsilon_k}-M_{\lambda^\delta}-M_{(\lambda-\delta-\epsilon_k)^\delta}.
\end{equation}

Notice that by Theorem 5.3,
\begin{equation}
L_{\lambda-\delta-\epsilon_k}=M_{\lambda-\delta-\epsilon_k}-
M_{(\lambda-\delta-\epsilon_k)^\delta}+L_{(\lambda-\delta-\epsilon_k)^\delta}.
\end{equation}
By (5.88) and (5.89),
\begin{equation}
L_\lambda=M_\lambda-M_{\lambda^\delta}+(1-x_1)L_{\lambda-\delta-\epsilon_k}-
(x_0+1)L_{(\lambda-\delta-\epsilon_k)^\delta}.
\end{equation}

Using the method introduced in Remark 4.8, we can get from (5.90)
that
\begin{eqnarray} &&
L_{\lambda+\sum_{i=k+1}^{m}\epsilon_i}=
M_{\lambda+\sum_{i=k+1}^{m}\epsilon_i}
-M_{(\lambda+\sum_{i=k+1}^{m}\epsilon_i)^\delta}\\\nonumber&&\quad\quad\quad
+(1-x_1)L_{\lambda-(m-k+1)\delta-\epsilon_k}
-(x_0+1)L_{(\lambda-(m-k+1)\delta-\epsilon_k)^\delta}.
\end{eqnarray}
If we multiply $L_\delta$ on both sides of (5.91) and select the
terms $L_\nu$ and $M_\nu$ with $\nu\sim\lambda$, then we get
\begin{eqnarray} &&
L_{\lambda+\sum_{i=k+1}^{m}\epsilon_i}=
M_{\lambda+\sum_{i=k+1}^{m}\epsilon_i}
-M_{(\lambda+\sum_{i=k+1}^{m}\epsilon_i)^\delta}\\\nonumber&&\quad\quad\quad
+(1-x_1)L_{(\lambda-(m-k+1)\delta-\epsilon_k)^\delta}
-(x_0+1)L_{\lambda-(m-k+1)\delta-\epsilon_k}.
\end{eqnarray}
Comparing (5.91) with (5.92) shows us that
\begin{equation}
x_1-1=x_0+1\quad \Rightarrow \quad x_1=2+x_0\geq 2.
\end{equation}
Inequalities (5.86) and (5.93) induce that $x_1=2$ and $x_0=0$.

Hence
\begin{equation}
L_\lambda=M_\lambda-M_{\lambda^\delta}-L_{\lambda-\delta-\epsilon_k}
-L_{(\lambda-\delta-\epsilon_k)^\delta}.
\end{equation}
By the way, we have got
\begin{equation}
[L_\lambda\otimes L_\delta: L_{\lambda-\delta}]=[L_\lambda\otimes
L_\delta: L_{\lambda-\epsilon_k}]=0.
\end{equation}

\noindent{\bf 2). For $\lambda_k=2$.}

Now (5.82) becomes to
\begin{equation}
L_\lambda+x_1L_{\lambda-\delta-\epsilon_k}+x_2L_{\lambda-2\delta-2\epsilon_k}
+x_0L_{(\lambda-2\delta-2\epsilon_k)^\delta}=
M_\lambda+M_{\lambda-\delta-\epsilon_k}-M_{\lambda^\delta}-M_{(\lambda-\delta-\epsilon_k)^\delta},
\end{equation} where $L_{\lambda-\delta-\epsilon_k}$ can be
expressed as form (5.94). Hence
\begin{equation}
L_\lambda=M_\lambda-M_{\lambda^\delta}+(1-x_1)L_{\lambda-\delta-\epsilon_k}
+(1-x_2)L_{\lambda-2\delta-2\epsilon_k}+(1-x_0)L_{(\lambda-2\delta-2\epsilon_k)^\delta}.
\end{equation}

Using the method introduced in Remark 4.8 again, we can get from
(5.97) that
\begin{eqnarray} \ &&
L_{\lambda+\sum_{i=k+1}^{m}\epsilon_i}=
M_{\lambda+\sum_{i=k+1}^{m}\epsilon_i}
-M_{(\lambda+\sum_{i=k+1}^{m}\epsilon_i)^\delta}+(1-x_1)L_{\lambda-\delta-\epsilon_k+\sum_{i=k+1}^{m}\epsilon_i}\\\nonumber&&\quad\quad\quad
+(1-x_2)L_{\lambda-(m-k+2)\delta-2\epsilon_k}
+(1-x_0)L_{(\lambda-(m-k+2)\delta-2\epsilon_k)^\delta}.
\end{eqnarray}

Multiply $L_\delta$ on the both sides of (5.98) and select the terms
$L_\nu$ and $M_\nu$ with $\nu\sim\lambda$, then we have
\begin{eqnarray}\ &&
L_{\lambda+\sum_{i=k+1}^{m}\epsilon_i}=
M_{\lambda+\sum_{i=k+1}^{m}\epsilon_i}
-M_{(\lambda+\sum_{i=k+1}^{m}\epsilon_i)^\delta}+(1-x_1)L_{\lambda-\delta-\epsilon_k+\sum_{i=k+1}^{m}\epsilon_i}\\\nonumber&&\quad\quad\quad
+(1-x_2)L_{(\lambda-(m-k+2)\delta-2\epsilon_k)^\delta}
+(1-x_0)L_{\lambda-(m-k+2)\delta-2\epsilon_k}.
\end{eqnarray}
Compare (5.98) and (5.99), then we get $x_2=x_0$.

Now we begin to determine $x_2$. Notice that
$\lambda-\delta-2\epsilon_k$ is a typical weight, hence
\begin{equation}
[L_\lambda\otimes L_\delta:
L_{\lambda-\delta-2\epsilon_k}]=[L_{(\lambda-2\delta-2\epsilon_k)^\delta}\otimes
L_\delta:L_{\lambda-\delta-2\epsilon_k}]=0
\end{equation} by Lemma 4.6,
\begin{equation}
[(M_\lambda-M_{\lambda^\delta})\otimes
L_\delta:L_{\lambda-\delta-2\epsilon_k}]=0
\end{equation} by direct calculation,
\begin{equation}
[L_{\lambda-\delta-\epsilon_k}\otimes
L_\delta:L_{\lambda-\delta-2\epsilon_k}]=0
\end{equation} by (5.95), and
\begin{equation}
[L_{\lambda-2\delta-2\epsilon_k}\otimes
L_\delta:L_{\lambda-\delta-2\epsilon_k}]=1
\end{equation} by Lemma 4.5.

Therefore (5.100)-(5.103) induce that $x_2=1$ (hence $x_0=1$, too).

We have known that $x_1=1$ or $2$. But it is impossible that $x_1=1$
because of Corollary 2.5. So it must be that $x_1=2$ and
\begin{equation}
L_\lambda=M_\lambda-M_{\lambda^\delta}-L_{\lambda-\delta-\epsilon_k}.
\end{equation}
It is easy to get from (5.104) that
\begin{equation}
[L_\lambda\otimes L_\delta: L_{\lambda-\delta}]=[L_\lambda\otimes
L_\delta: L_{\lambda-\epsilon_k}]=0.
\end{equation}

\noindent{\bf 3). For any $\lambda_k$.}

We shall use induction on $\lambda_m$ to show that (5.104) and
(5.105) also hold for any $\lambda_k\geq2$.

By induction assumption and (5.82),
\begin{equation}
L_\lambda=M_\lambda-M_{\lambda^\delta}+\sum_{i=1}^2(1-x_i)L_{\lambda-i\delta-i\epsilon_k}-
\sum_{i=3}^{\lambda_k}x_iL_{\lambda-i\delta-i\epsilon_k}
-x_0L_{(\lambda-\lambda_k\delta-\lambda_k\epsilon_k)^\delta}.
\end{equation}
Multiply $L_\delta$ on both sides of (5.106). Then if we choose the
terms in the block corresponding to
$\lambda-i\delta-(i+1)\epsilon_k$ ($1\leq i\leq \lambda_k-1$), there
comes $x_2=1$, $x_3=\cdots=x_{\lambda_k}=0$. Imitating
(5.91)-(5.93), we can get $x_0=x_{\lambda_k}=0$.

Thus
\begin{equation}
L_\lambda=M_\lambda-M_{\lambda^\delta}+(1-x_1)L_{\lambda-\delta-\epsilon_k}.
\end{equation}
Recall that we have known $x_1=1$ or $2$.  Thanks to Corollary 2.5,
it should be that $x_1=2$. Hence we get that (5.104) and (5.105)
still hold for any $\lambda_k\geq2$.

\subsubsection{\bf The case that $\lambda$ is
$(\delta+\epsilon_k)$-atypical} Thanks to Remark 4.8, now there is
no difficulty for us to get
$L_{(2m-k+\lambda_k-1)\delta+\sum_{i=1}^{m}\lambda_i\epsilon_i}$
from
$L_{(2m-k+\lambda_k-1)\delta+\sum_{i=1}^{k}\lambda_i\epsilon_i}$. We
shall state the theorem in the next section together with the case
$\mathfrak{g}=\mathfrak{osp}(2m|2)$.

\subsection{Final results for the case of $\lambda^\delta\not\in\mathcal {P}$}
For any
$\lambda=\lambda_0\delta+\sum_{i=1}^{m}\lambda_i\epsilon_i\in\mathcal
{P}$, if it is $(\delta+\epsilon_k)$-atypical and $\lambda_k\neq0$,
then we define
\begin{equation}
\varphi(\lambda)=\left\{\begin{array}{lc}\lambda-(m-k+1)\delta-\sum_{i=k}^{m-1}\epsilon_i+\epsilon_m
,& \mbox{if $\mathfrak{g}=\mathfrak{osp}(2m|2)$ and
$\lambda_m=-\lambda_k$;}
\\\lambda-(t-k+1)\delta-\sum_{i=k}^t\epsilon_i, & \mbox{otherwise,} \end{array}\right.
\end{equation} where $t\geq k$ is the maximal number such that
$|\lambda_t|=\lambda_k$.

when $\mathfrak{g}=\mathfrak{osp}(2m|2)$, if $\lambda$ is
$(\delta-\epsilon_m)$-atypical and $\lambda_m\neq0$, we can also
define
\begin{equation}
\varphi(\lambda)=\lambda-\delta+\epsilon_k.
\end{equation}

It is easy to show that $\lambda\sim\varphi(\lambda)\in\mathcal
{P}$.

Denote by $\theta_\lambda$ the maximal integer such that
$\varphi^{\theta_\lambda}(\lambda)$ is well defined. It is clear
that $\varphi^{\theta_\lambda}(\lambda)$ and
$\varphi^{\theta_\lambda}(\lambda)^\delta$ are both in $\mathcal
{P}$. Denote $\lambda^T=\varphi^{\theta_\lambda}(\lambda)^\delta$
(hence $\varphi^{\theta_\lambda}(\lambda)=(\lambda^T)^\delta$) and
call it a \emph{tail typical} weight.

Using the method introduced in Remark 4.8, we can complete our
argument in subsection 5.3 and 5.4 by the above notations. To
summarize:
\begin{theorem}
For any atypical weight $\lambda\in\mathcal {P}$ with
$\lambda^\delta\not\in\mathcal {P}$,
\begin{equation}
L_\lambda=\left\{
\begin{array}{ll}
M_\lambda-M_{\lambda^\delta}-L_{\varphi(\lambda)},& \mbox{if
$\varphi(\lambda)^\delta\not\in\mathcal {P}$ or $\varphi(\lambda)=\varphi(\lambda)^\delta$}\\
M_\lambda-M_{\lambda^\delta}-L_{\varphi(\lambda)}-L_{\varphi(\lambda)^\delta},
&\mbox{if $\varphi(\lambda)^\delta\in\mathcal {P}$ and
$\varphi(\lambda)\neq\varphi(\lambda)^\delta$}.
\end{array}\right.
\end{equation}
More precisely, if $\mathfrak{g}=\mathfrak{osp}(2m|2)$ and
${(\lambda^T)}_0=m-1$ (i.e.
$\varphi^{\theta_\lambda}(\lambda)=\lambda^T$), then
\begin{eqnarray}
L_\lambda&=&\sum_{i=0}^{\theta_\lambda-1}(-1)^i(M_{\varphi^i(\lambda)}-M_{\varphi^i(\lambda)^\delta})
+(-1)^{\theta_\lambda}M_{\lambda^T}
\\\nonumber&&
+\sum_{q=(\lambda^T)_{1}+1}^{\infty}(-1)^{\theta_\lambda+q}
(M_{(\lambda^T)^{0,q}} +M_{(\lambda^T)^{0,q}_{-}})+\\\nonumber
&&\sum_{j=1}^k\sum_{q=(\lambda^T)_{j+1}+1}^{(\lambda^T)_j}(-1)^{\theta_\lambda+q}
(M_{(\lambda^T)^{j,q}} +M_{(\lambda^T)^{j,q}_{-}});
\end{eqnarray}
otherwise,
\begin{eqnarray}
L_\lambda=\sum_{i=0}^{\theta_\lambda-1}(-1)^i(M_{\varphi^i(\lambda)}-M_{\varphi^i(\lambda)^\delta})
+(-1)^{\theta_\lambda}(M_{\lambda^T}+M_{(\lambda^T)^\delta})+\\\nonumber
2\sum_{q=(\lambda^T)_{1}+1}^{\infty}(-1)^{\theta_\lambda+q}M_{(\lambda^T)^{0,q}}+
2\sum_{j=1}^k\sum_{q=(\lambda^T)_{j+1}+1}^{(\lambda^T)_j}(-1)^{\theta_\lambda+q}M_{(\lambda^T)^{j,q}}.
\end{eqnarray}
\end{theorem}

\subsection{Decomposition of $L_\lambda\otimes L_\delta$}

In our arguments before, we also built up enough information about
the tensor module $L_\lambda\otimes L_\delta$. As a co-product, its
decomposition has been obtained.
\begin{corollary}
a) For any atypical weight $\lambda\in \mathcal {P}$, the tensor
module $L_\lambda\otimes L_\delta$ is completely reducible.
Precisely, if there exists $1\leq i\leq m$ such that
$(\lambda+\rho,\delta+\epsilon_i)=0$ and
$\lambda-(\delta+\epsilon_i)\in \mathcal {P}$ , then
\begin{equation}
L_\lambda\otimes L_\delta=\bigoplus_{\mu\in
P^+_\lambda\setminus\{\lambda-\delta,\lambda-\epsilon_i\}}L_\mu.
\end{equation}
When $\mathfrak{g}=\mathfrak{osp}(2m|2)$, if
$(\lambda+\rho,\delta-\epsilon_m)=0$ and
$\lambda-(\delta-\epsilon_m)\in \mathcal {P}$ , then
\begin{equation}
L_\lambda\otimes L_\delta=\bigoplus_{\mu\in
P^+_\lambda\setminus\{\lambda-\delta,\lambda+\epsilon_m\}}L_\mu.
\end{equation}
Otherwise,
\begin{equation}
L_\lambda\otimes L_\delta=\bigoplus_{\mu\in \mathcal
{P}_\lambda}L_\mu.
\end{equation}

b) For any typical weight $\lambda\in \mathcal {P}$, the tensor
module $L_\lambda\otimes L_\delta$ can be written as
\begin{equation}
L_\lambda\otimes L_\delta=\sum_{\mu\in \mathcal
{P}_\lambda}\left(a_\mu
L_\mu+{a_{\varphi(\mu)}L_{\varphi(\mu)}}+a_{\varphi(\mu)^{\delta}}L_{\varphi(\mu)^{\delta}}\right),
\end{equation}
where $a_\mu=2$ if there exists $\mu'\in \mathcal {P}_\lambda$ with
$\mu'\succ \mu$ and $a_\mu=1$ otherwise, $a_{\varphi(\mu)}=1$ (resp.
$a_{\varphi(\mu)^{\delta}}=1$) if $a_\mu=2$,
$\varphi(\mu)\in\mathcal {P}$ (resp. $\varphi(\mu)^\delta\in\mathcal
{P}$) and $a_{\varphi(\mu)}=0$ (resp. $a_{\varphi(\mu)^{\delta}}=0$)
otherwise. furthermore, there are at most two $\mu$'s with
$a_\mu=2$. When the $a_\mu$'s are all equal to $1$, the tensor
module is completely reducible.
\end{corollary}

\section*{Acknowledgement} Part of the work was done when the author
visited the University of Sydney from January to June, 2009. He
would like to express his deep thanks to Prof. Ruibin Zhang, without
whom the work can not be finished. He is also indebted to the
Academy of Mathematics and Systems Science, Chinese Academy of
Sciences for its financial support during that time.

\bibliographystyle{amsplain}

\end{document}